\documentclass[a4paper, 12pt]{article}

\usepackage[T1]{fontenc}
\usepackage{titlesec}
\usepackage{datetime}
\usepackage[english]{babel}
\usepackage{geometry}
\usepackage{caption}
\usepackage{subcaption}

\usepackage{bm}
\usepackage{amssymb}
\usepackage{amsmath}
\usepackage{amsthm}
\usepackage[dvipdfm]{graphicx} 
\usepackage{verbatim,enumerate}

\usepackage[active]{srcltx}
\usepackage{array}
\newcolumntype{L}{>{\displaystyle}l}
\newcolumntype{C}{>{\displaystyle}c}
\newcolumntype{R}{>{\displaystyle}r}

 \newtheorem{theorem}{Theorem}[section]

 \newtheorem{lemma}[theorem]{Lemma}
 \newtheorem{corollary}[theorem]{Corollary}
\theoremstyle{definition}
 \newtheorem{definition}[theorem]{Definition}

 \newtheorem{example}[theorem]{Example}
\numberwithin{equation}{section}

\usepackage[usenames]{color}

\newcommand{\R}{\boldsymbol{R}}

\newcommand{\rank}{\operatorname{rank}}

\renewcommand{\phi}{\varphi}

\newcommand{\A}{\mathcal{A}}

\newcommand{\mycomment}[1]{}
\renewcommand{\tilde}{\widetilde}

\begin{document}

\title{Geometry on the Gluing Locus of Two Surfaces}
\author{Li Junzhen}
\date{\today}
\maketitle

\begin{quote}
In this paper, we deal with the gluing of two surfaces, where the gluing locus is assumed to be a curve. We consider a moving frame along the gluing locus,
and define developable surfaces with respect to the frame.
Considering geometric properties of these developable surfaces, we study the geometry of gluing two surfaces.
\end{quote}

\section{Introduction}
In recent decades, with the development of computer graphics, discrete surface theory has been studied by many authors. In this theory, the gluing of two surfaces is important. In this paper, we pay attention to where two surfaces are glued together. That set is called the gluing locus and we assume it is a curve. We define a frame along the gluing locus by using each normal vector on each surface that is glued. Once a frame along a curve is given, then developable surfaces are defined naturally. And these surfaces represent the geometric properties of the frame. It is known that developable surfaces in the three-dimensional Euclidean space are classified into cylinders, cones and tangent developable surfaces. Among these, cylinders and cones are special ones. Each condition for the developable surfaces defined by the frame to be a cylinder or a cone should be considered as a special gluing. Furthermore, singularities of the developable surfaces defined by the frame should represent geometric properties of gluing. We study geometry on the gluing locus of two glued surfaces
by considering the above cases.

\section{Preliminaries}
We prepare the necessary notations and
organize the geometry of ruled surfaces, 
frontal surfaces and their singularities.

\subsection{Ruled surfaces, developable surfaces, and their singularities}
In this section, we deal with ruled surfaces, 
developable surfaces and their singularities.
For more details, see \cite{gray,izdev,iztakeruled}.
Let $I \subset \R$ be an open interval, 
$c:I \to \R^3$ be a curve
and $\delta:I \to \R^3$ be a curve such that $|\delta|=1$.
The surface defined by
$$
r(t,a)=c(t)+a\delta(t)
$$
is called a {\it ruled surface}.
\begin{definition}
	A ruled surface $r(t,a)=c(t)+a\delta(t)$ is a {\it cylinder}
	if $\delta'=0$ holds for any $t \in I$.
	Then it is said to be {\it non-cylindrical} if
	$$
	\delta'\ne0
	$$
	holds for any $t \in I$, where $'=\partial/\partial t$.
\end{definition}
For a ruled surface, there is a curve
which is called a {\it striction curve}.
\begin{definition}
	For a non-cylindrical ruled surface $r(t,a)=c(t)+a\delta(t)$,
	the curve $\hat{\sigma}(t)=r(t,s(t))$
	is called a {\it striction curve} of $r$ if
	$$
	\hat\sigma'\cdot \delta' = 0,
	$$
	holds for any $t \in I$.
\end{definition}
If $r$ is a cylinder, since $\delta' = 0$ holds, 
for any function $s(t)$, the curve $r(t, s(t))$ is a striction curve.
Since $\delta' \cdot \delta' \ne 0$ holds
for a non-cylindrical ruled surface $r$,
setting
$$
s(t) = -\frac{c' \cdot \delta'}{\delta' \cdot \delta'}.
$$
The curve $\hat\sigma(t)=r(t,s(t))$ is a {\it striction curve} of $r$.
Moreover, it is known that for a non-cylindrical ruled surface $r$,
the set of singular points is included by the image of
the striction curve $\hat{\sigma}(t)$.
\begin{definition}
	A ruled surface $ r $ is called a {\it cone} 
	if it is non-cylindrical.
	Moreover, the image of the striction curve $\hat{\sigma}$ is a single point.
\end{definition}
A ruled surface with zero Gaussian curvature is called
a {\it developable surface}.
It is known that developable surfaces are classified as cylinder, cone, tangent developable, or combinations of them.

\subsection{Frame along curve on frontal}\label{sec:frame}
In this section, we introduce the notion of a front
which is a singular surface with well-defined
unit normal vector.
Let $U\subset\R^2$ be an open set.
\begin{definition}
	A map $f: U \to \R^3$ is called a {\it frontal} 
	if there exists a map $\nu: U \to \R^3$ 
	such that $|\nu| = 1$ 
	and for any point $p \in U$ 
	and any vector $X \in T_p \R^2$, 
	the condition
	$$
	df_p(X) \cdot \nu(p) = 0
	$$
	holds. 
	The map $\nu$ 
	is called the {\it unit normal vector} of $f$.
	A frontal $f$ is called a {\it front} 
	if $(f, \nu)$ is an immersion.
\end{definition}
If $f$ is a regular surface, then $\nu$ can be taken as the usual unit normal vector. Therefore, a regular surface is a frontal. 
Moreover, one can easily see that $(f,\nu)$ is an immersion.
Thus, a regular surface is a front.
On the other hand, a frontal can have singularities.
Let $f: U \rightarrow \R^3$ be a frontal.
Let $I \subset \R$ be an open interval,
$\gamma: I \to U$ be a curve.
We set
$$
\tilde{\gamma}(t) = f \circ \gamma(t).
$$
Then this is a curve on $f$.
We do not assume $t$ is an arc-length parameter, 
then $\gamma$ may have singular points.
We assume that there exits a function $l(t)$ and 
a unit vector $\boldsymbol{e}$ satisfying 
$\tilde{\gamma}' = l\boldsymbol{e}$.
We take $\boldsymbol{\nu}$ which is a unit vector field
along $\tilde{\gamma}$ normal to $\boldsymbol{e}$.
We set $\boldsymbol{b} = \boldsymbol{e}\times\boldsymbol{\nu}$.
Then we have the frame
$\{\boldsymbol{e}, \boldsymbol{\nu}, \boldsymbol{b}\}$ 
along $\tilde{\gamma}$ on the frontal $f$. 
Here, since $\boldsymbol{\nu}$ 
is not necessarily the restriction of the unit normal vector of $f$,
the frame taken here is not necessarily the Darboux frame.
In this case, the functions $\kappa_1(t)$, $\kappa_2(t)$ and $\kappa_3(t)$ are
defined by the following Frenet-Serret type formulas:
\begin{equation}\label{eq:frenet}
	\begin{bmatrix}
		\boldsymbol{e} \\
		\boldsymbol{\nu} \\
		\boldsymbol{b}\\
	\end{bmatrix}'
	= 
	\begin{bmatrix}
		0 & \kappa_1 & \kappa_2 \\
		-\kappa_1 & 0 & \kappa_3 \\
		-\kappa_2 & -\kappa_3 & 0 \\
	\end{bmatrix}
	\begin{bmatrix}
		\boldsymbol{e} \\
		\boldsymbol{\nu} \\
		\boldsymbol{b}\\
	\end{bmatrix}.
\end{equation}
If the frame is a Darboux frame, then $t$ is an arc-length.
Therefore, $\kappa_1$ is the {\it normal curvature}, 
$\kappa_2$ is the {\it geodesic curvature}, 
and $\kappa_3$ is the {\it geodesic torsion}, 
which are invariants of the curve $\hat{\gamma}$ on the frontal $f$. 
However, since the frame is not necessarily a Darboux frame, 
these functions are not necessarily equal to these curvatures.

\subsection{Developable surfaces along curves}\label{sec:dev}
Let a curve $\tilde{\gamma}$ and 
a frame $\{\boldsymbol{e}, \boldsymbol{\nu}, \boldsymbol{b}\}$ 
along $\tilde{\gamma}$ 
be given as in Section \ref{sec:frame}.
In this case, 
a developable surface along $\tilde{\gamma}$ 
can be defined as the following way (See \cite{izohflat}, for example.).
For a unit vector field $\boldsymbol{v}$ along $\tilde{\gamma}$, 
we define the function $H_{\boldsymbol{v}}:I \times \R^3 \to \R$ 
by
$$
H_{\boldsymbol{v}}(t, X) = 
\boldsymbol{v}(t) \cdot \big(X-\tilde\gamma(t)\big).
$$
This is called the {\it height function with respect to $\boldsymbol{v}$}.
Furthermore, let us set $h_{\boldsymbol{v}}(t) = H_{\boldsymbol{v}}(t, 0)$. 
The function $H_{\boldsymbol{v}}$ can be interpreted as a 3-parameter family of 1-variable functions.
For each $t \in I$, the set
$$
\mathcal{H}_{\boldsymbol{v}}=\{X \in \R^3 \mid H_{\boldsymbol{v}}(t, X) = 0\}
$$
is a plane orthogonal to $\boldsymbol{v}$.
Thus $\mathcal{H}_{\boldsymbol{v}}$ is a 1-parameter family of planes.
Consider the envelope of this family of planes
$$
\mathcal{D}_{\boldsymbol{v}} 
=
\{X \in \R^3 \mid \text{ there exists }t \in I \text{ such that } 
H_{\boldsymbol{v}}(t, X) = H_{\boldsymbol{v}}'(t, X) = 0\}.
$$
However, for the case of $\boldsymbol{v}=\boldsymbol{e}$, 
it is a family of normal planes to $\tilde{\gamma}$, 
which has no meaning as a curve on the surface. 
Therefore, here we consider the cases for
$\boldsymbol{v}=\boldsymbol{\nu}$ and $\boldsymbol{v}=\boldsymbol{b}$.
Then we get two envelopes 
$\mathcal{D}_{\boldsymbol{\nu}}$ and $\mathcal{D}_{\boldsymbol{b}}$
constructed from the same curve.
The following holds.

\begin{lemma}\label{lem:svsb}
	We assume that $(\kappa_1,\kappa_3)\neq(0,0)$ on $I$.
	Then we set a ruled surface $S_{\boldsymbol{\nu}}(t,a)$.
	We assume that $(\kappa_2,\kappa_3)\neq(0,0)$ on $I$.
	Then we set a ruled surface $S_{\boldsymbol{b}}(t,a)$.
	Here, $S_{\boldsymbol{\nu}}(t,a)$ and $S_{\boldsymbol{b}}(t,a)$ are
	given by
	\begin{align*}
		S_{\boldsymbol{\nu}}(t,a) 
		&=
		\tilde{\gamma}(t) + a \delta_{\boldsymbol{\nu}}(t),\quad
		\left(\delta_{\boldsymbol{\nu}}(t)
		=
		\dfrac{\kappa_3 \boldsymbol{e} + \kappa_1 \boldsymbol{b}}
		{\sqrt{\kappa_3^2+\kappa_1^2}}\right); \\
		S_{\boldsymbol{b}}(t,a) 
		&= 
		\tilde{\gamma}(t) + a \delta_{\boldsymbol{b}}(t),\quad
		\left(\delta_{\boldsymbol{b}}(t)
		=
		\dfrac{\kappa_3 \boldsymbol{e} - \kappa_2 \boldsymbol{\nu}}
		{\sqrt{\kappa_3^2+\kappa_2^2}}\right).
	\end{align*}
	Under each assumption,
	the image of $S_{\boldsymbol{\nu}}$
	coincides with the set $\mathcal{D}_{\boldsymbol{\nu}}$,
	and the image of $S_{\boldsymbol{b}}$ 
	coincides with the set $\mathcal{D}_{\boldsymbol{b}}$.
\end{lemma}

\begin{proof}
	We show the case $\boldsymbol{v}=\boldsymbol{\nu}$.
	By the condition $H_{\boldsymbol{\nu}}(t, X) = 0$, 
	it holds that there exist $c_1,c_2 \in \R$ such that
	$X-\hat{\gamma}(t) 
	= c_1 \boldsymbol{e} + c_2 \boldsymbol{b}.$
	By \eqref{eq:frenet} and $\hat{\gamma}'=l(t)\boldsymbol{e}$.
	Moreover, substituting $X-\hat{\gamma}(t)$ into the formula
	$H_{\boldsymbol{\nu}}'(t, X) = 0$, we get
	\begin{align*}
		H_{\boldsymbol{\nu}}'(t, X) 
		&= \boldsymbol{\nu}' \cdot \big(X-\hat{\gamma}\big) + \boldsymbol{\nu} \cdot \big(X-\hat{\gamma}\big)'\\
		&= \boldsymbol{\nu}' \cdot \big(X-\hat{\gamma}\big) + \boldsymbol{\nu} \cdot \big(X'-\hat{\gamma}'\big)\\
		&= \boldsymbol{\nu}' \cdot \big(X-\hat{\gamma}\big) + \boldsymbol{\nu} \cdot \big(0-l\boldsymbol{e}\big)\\
		&= \boldsymbol{\nu}' \cdot \big(X-\hat{\gamma}\big)\\
		&= (-\kappa_1 \boldsymbol{e} + \kappa_3 \boldsymbol{b}) \cdot (c_1 \boldsymbol{e} + c_2 \boldsymbol{b}) \\
		&= -c_1 \kappa_1 + c_2 \kappa_3\\
		&= 0.
	\end{align*}
	Thus set
	$X-\hat{\gamma}(t) 
	= a \left(\kappa_3\boldsymbol{e} + \kappa_1 \boldsymbol{b}\right)$, where $a=c_1/\kappa_3\in\R$.
	Hence the image of 
	$$
	X(t,a)
	=\tilde{\gamma}(t) + a \left(\kappa_3\boldsymbol{e} + \kappa_1 \boldsymbol{b}\right)
	$$ coincides with 
	$\mathcal{D}_{\boldsymbol{\nu}}$.
	We set
	$\delta_{\boldsymbol{\nu}}(t)=x/|x|$
	and $S_{\boldsymbol{\nu}}(t,a)=\hat{\gamma}+a\delta_{\boldsymbol{\nu}}$, where
	$x=\kappa_3\boldsymbol{e} + \kappa_1 \boldsymbol{b}$,
	by the assumption $(\kappa_1,\kappa_3)\neq(0,0)$.
	Obviously, the image of $S_{\boldsymbol{\nu}}(t,a)$ is the same as that of $X(t,a)$, and therefore the same as that of $\mathcal{D}_{\boldsymbol{\nu}}$.
	Thus we can get the conclusion.
	We can show the case of $\boldsymbol{v}=\boldsymbol{b}$ by a similar calculation.
\end{proof}

\begin{lemma}\label{lem:s2s3frl}
	Both surfaces $S_{\boldsymbol{\nu}}$ and $S_{\boldsymbol{b}}$ are frontals.
	In particular, 
	$\boldsymbol{\nu}$ can be taken as the unit normal vector of $S_{\boldsymbol{\nu}}$, 
	and $\boldsymbol{b}$ can be taken as the unit normal vector of $S_{\boldsymbol{b}}$.
\end{lemma}
\begin{proof}
	Let $(S_i)_t=\partial S_i/\partial t$ and 
	$(S_i)_a=\partial S_i/\partial a$, 
	then we see
	\begin{align}
		(S_{\boldsymbol{\nu}})_t(t,a) & = 
		\Big(
		l
		+
		a\big(
		\dfrac{\kappa_3'-\kappa_1\kappa_2}{\sqrt{\kappa_3^2+\kappa_1^2}}
		-\dfrac{\kappa_3(\kappa_1\kappa_1'+\kappa_3\kappa_3')}
		{\sqrt{\kappa_3^2+\kappa_1^2}^3}\big)
		\Big)\boldsymbol{e}\label{eq:s2ta}
		\\
		&\hspace{20mm}+
		a\Big(
		\dfrac{\kappa_1'+\kappa_2\kappa_3}{\sqrt{\kappa_3^2+\kappa_1^2}}
		-
		\dfrac{\kappa_1(\kappa_1\kappa_1'+\kappa_3\kappa_3')}
		{\sqrt{\kappa_3^2+\kappa_1^2}^3}
		\Big)\boldsymbol{b}\nonumber \\
		(S_{\boldsymbol{b}})_t(t,a)&=
		\Big(l+a\big(\dfrac{\kappa_3'+\kappa_1\kappa_2}
		{\sqrt{\kappa_3^2+\kappa_2^2}}
		-\dfrac{\kappa_3(\kappa_2\kappa_2'+\kappa_3\kappa_3')}
		{\sqrt{\kappa_3^2+\kappa_2^2}^3}\big)\Big)
		\boldsymbol{e}\label{eq:s3ta}\\
		&\hspace{20mm}
		+a\Big(\dfrac{\kappa_1\kappa_3-\kappa_2'}{\sqrt{\kappa_3^2+\kappa_2^2}}+\dfrac{\kappa_2(\kappa_2\kappa_2'+\kappa_3\kappa_3')}{\sqrt{\kappa_3^2+\kappa_2^2}^3}\Big)\boldsymbol{\nu},  \nonumber
	\end{align}
	and
	$(S_{\boldsymbol{\nu}})_a(t)=\delta_{\boldsymbol{\nu}}$, $(S_{\boldsymbol{b}})_a(t) =\delta_{\boldsymbol{b}}$.
	Therefore, 
	the unit normal vector of $S_{\boldsymbol{\nu}}$ 
	can be taken as $\boldsymbol{\nu}$, and
	the unit normal vector of $S_{\boldsymbol{b}}$ 
	can be taken as $\boldsymbol{b}$.
	Where $l(t)$ is the function 
	satisfying $\tilde{\gamma}'(t)=l\boldsymbol{e}$.
\end{proof}
Moreover, the following holds.
\begin{lemma}
	The frontals $S_{\boldsymbol{\nu}}$ and $S_{\boldsymbol{b}}$ 
	are developable surfaces.
\end{lemma}
\begin{proof}
	As seen above, 
	the unit normal vector of $S_{\boldsymbol{\nu}}$ can be taken as $\boldsymbol{\nu}$,
	and the unit normal vector of $S_{\boldsymbol{b}}$ can be taken as $\boldsymbol{b}$.
	Therefore, the derivatives of these vectors with respect to $a$ are zero. 
	This shows the assertion.
\end{proof}
\subsection{Singularities and their criteria}\label{sec:sing}
In this section, we consider only local properties
and we describe using the notion of germs.
For details, see \cite{krsuy, usybook}.
\begin{definition}
	Two map germs $f,g:(\R^2,0)\to(\R^3,0)$ 
	are called {\it $\A$-equivalent} 
	if there exist a diffeomorphism $\phi:(\R^2,0)\to(\R^2,0)$ of the domain 
	and a diffeomorphism $\Phi:(\R^3,0)\to(\R^3,0)$ of the codomain 
	such that
	$$
	\Phi\circ f\circ \phi^{-1}=g.
	$$
\end{definition}
The generic singularities of frontals are the following.
A map-germ $f$ is called a {\it cuspidal edge} 
if it is $\A$-equivalent to $(u,v)\mapsto(u,v^2,v^3)$ at the origin, as shown in the Figure \ref{fig:cusp}.
A map-germ $f$ is called a {\it swallowtail} 
if it is $\A$-equivalent to $(u,v)\mapsto(u,4v^3+2uv,3v^4+uv^2)$
at the origin, as shown in the Figure \ref{fig:swall}.
\begin{figure}[htbp]
	\centering
	\begin{minipage}[b]{0.50\textwidth}
		\centering		\includegraphics[width=4cm]{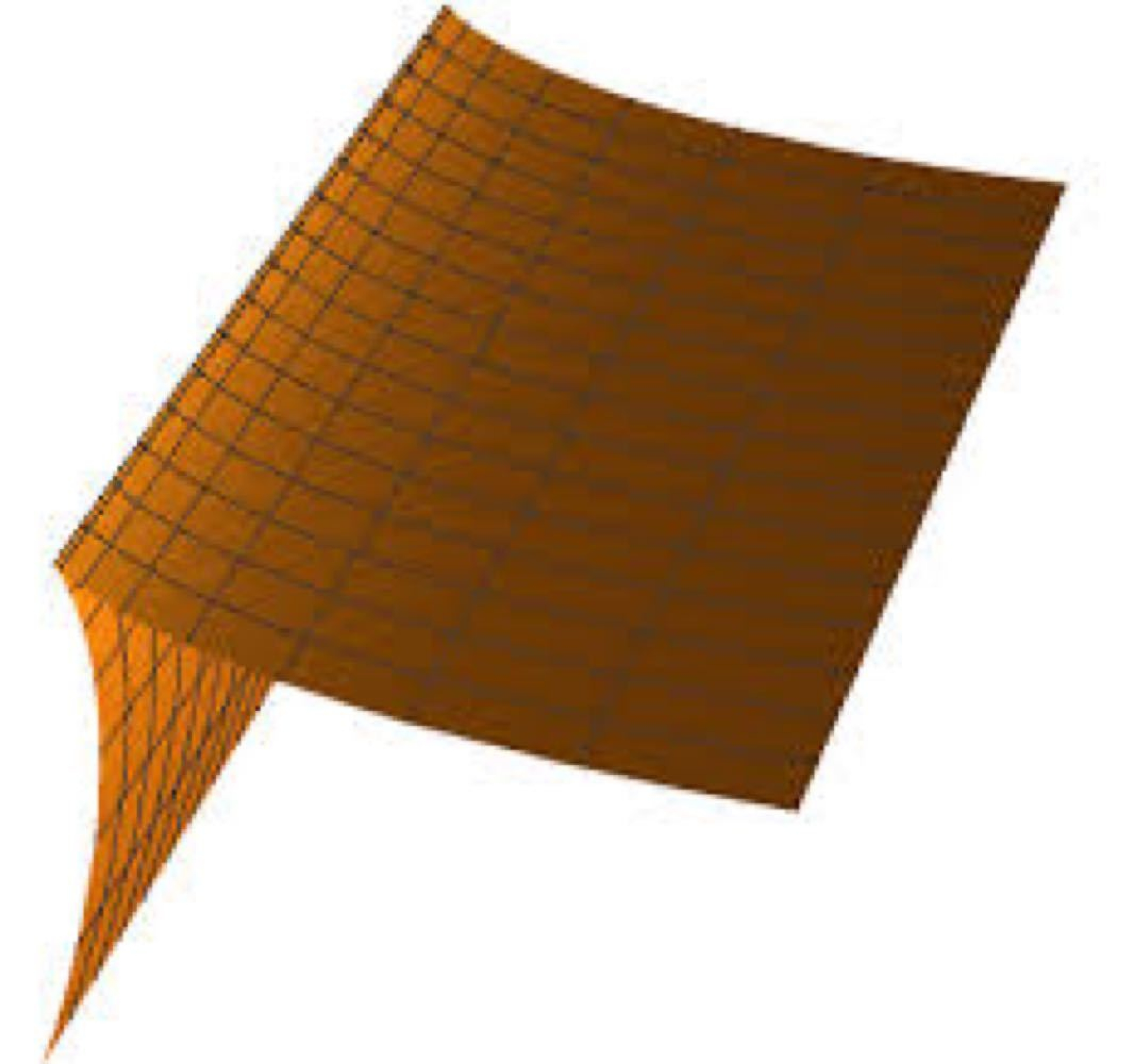}
		\caption{Cuspidal edge}
		\label{fig:cusp}
	\end{minipage}
	\hfill
	\begin{minipage}[b]{0.49\textwidth}
		\centering		\includegraphics[width=4cm]{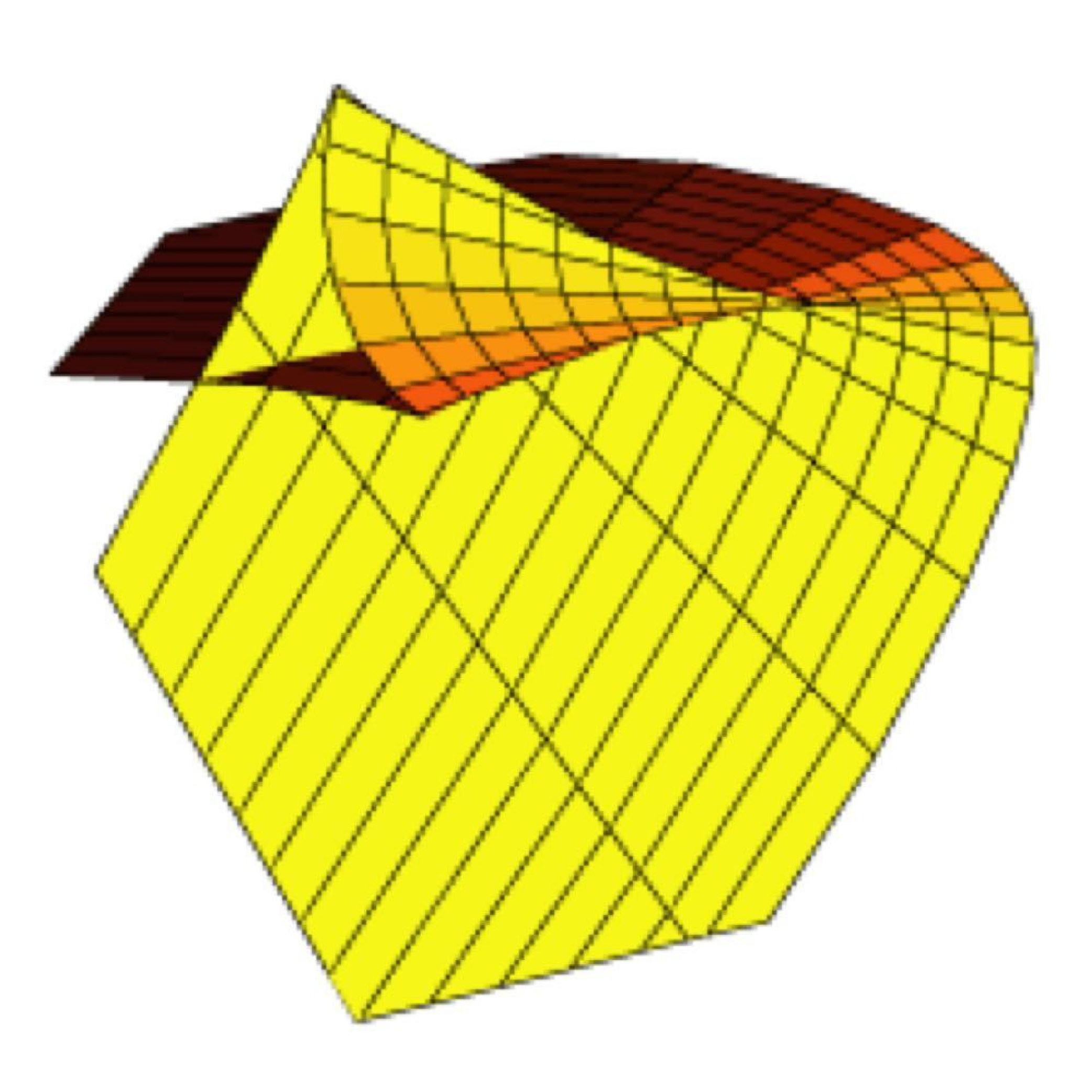}
		\caption{Swallowtail}
		\label{fig:swall}
	\end{minipage}
\end{figure}

There are useful methods to determine 
whether these singularities are of the types mentioned above.
Let $f:(\R^2,0)\to(\R^3,0)$ be a front. 
Let $\boldsymbol{\nu}$ be the unit normal vector of front $f$.
We take a coordinate system $(u,v)$.
\begin{definition}
	A function $\lambda$ 
	is called a {\it identifier of singularities} 
	if $\lambda$ is a non-zero scalar multiple of
	$$
	\lambda(u, v):=\det(f_u,f_v,\boldsymbol{\nu}).
	$$
\end{definition}
If $\lambda$ is a singularity identifier, 
then $\lambda^{-1}(0) = S(f)$, 
where $S(f)$ is the set of singularities of $f$.
A singularity $p \in S(f)$ of $f$ 
is called {\it non-degenerate} 
if $d\lambda_p \ne 0$. 
When $p$ is a non-degenerate singularity, 
$S(f)$ is a regular curve near $p$.
Let $f:(\R^2,0)\to(\R^3,0)$ be a front such that $\rm rank$ $df_0 = 1$,
there exists a vector field $\eta$ 
such that for any $p \in S(f)$, it holds that
$$
\ker df_p := \left\langle \eta_p \right\rangle_{\R}
$$
This $\eta$ is called a {\it null vector field}.
Criteria for cuspidal edges and swallowtails are given through the singularity identifier and the null vector field.
Then the following holds.
\begin{theorem}\cite{krsuy}
	Let $f:(\R^2,0)\to(\R^3,0)$ be a front with $\rank df_0=1$. 
	Let $\lambda$ be a singularity identifier 
	and $\eta$ be a null vector field.
	The front $f$ is a cuspidal edge if and only if 
	$$
	\eta\lambda(0)\ne0
	$$
	hold. The front $f$ is a swallowtail if and only if
	$$
	\eta\lambda(0)=0,\quad
	\eta\eta\lambda(0)\ne0,\quad
	d\lambda(0)\ne0
	$$
	hold.
\end{theorem}

\section{Geometry and Singularities of Surfaces $S_{\boldsymbol{\nu}}$ and $S_{\boldsymbol{b}}$}
In this section, we will describe the conditions that $S_{\boldsymbol{\nu}}(t,a)$ and $S_{\boldsymbol{b}}(t,a)$ obtained in Lemma \ref{lem:svsb} to be cylinder or cone and having singularities introduced in Section \ref{sec:sing}
in terms of the invariants $(\kappa_1,\kappa_2,\kappa_3)$ and the length function $l(t)$ of $\hat{\gamma}(t)$.

\subsection{Properties of the Surface $S_{\boldsymbol{\nu}}$}\label{sub:sv}
In this subsection, we assume $(\kappa_1,\kappa_3)\neq(0,0)$ 
for any $t\in I$.
By a direct calculation, we have
\begin{equation}\label{eq:betav}
	\delta_{\boldsymbol{\nu}}'(t)=
	\beta_{\boldsymbol{\nu}}\boldsymbol{w},\quad
	\left(\beta_{\boldsymbol{\nu}}(t)=
	\kappa_1^2\kappa_2+ \kappa_2\kappa_3^2+ \kappa_1'\kappa_3  - \kappa_1\kappa_3',\quad
	\boldsymbol{w}(t)=
	\frac{-\kappa_1\boldsymbol{e}+\kappa_3\boldsymbol{b}}{(\kappa_3^2+\kappa_1^2)^\frac{3}{2}}\right).
\end{equation}
If the developable surface $S_{\boldsymbol{\nu}}$ is non-cylindrical, then setting
\begin{equation}\label{eq:stnu}
	s(t) 
	= \frac{l\kappa_1\sqrt{\kappa_3^2+\kappa_1^2}}
	{\beta_{\boldsymbol{\nu}}},
\end{equation}
striction curve is obtained by $\hat{\sigma}_{\boldsymbol{\nu}}(t) = S_{\boldsymbol{\nu}}(t, s(t))$.
Under the assumption $(\kappa_1,\kappa_3)\neq(0,0)$.
The singular points of $S_{\boldsymbol{\nu}}$ 
satisfies that
$S(S_{\boldsymbol{\nu}}) = \{(t, a) \mid a = s(t)\}$.
We have the follows.
\begin{theorem}\label{thm:s3cylcone}
	$(1).$The developable surface $S_{\boldsymbol{\nu}}$ is a cylinder if and only if
	$$
	\beta_{\boldsymbol{\nu}}\equiv 0,
	$$
	where $\equiv$ stands for the equality holds identically.
	Similarly, $S_{\boldsymbol{\nu}}$ is non-cylindrical if 
	$\beta_{\boldsymbol{\nu}}$ never vanishes on $I$.
	
	$(2).$The developable surface $S_{\boldsymbol{\nu}}$ is a {\it cone} if and only if
	$\beta_{\boldsymbol{\nu}} \neq 0$ and $\rho_{\boldsymbol{\nu}}\equiv0$, where
	$$
	\rho_{\boldsymbol{\nu}}(t)
	=
	l\big(\beta_{\boldsymbol{\nu}}(\kappa_2\kappa_3+2\kappa_1')-\beta_{\boldsymbol{\nu}}'\kappa_1\big)+l'\kappa_1\beta_{\boldsymbol{\nu}}.
	$$
\end{theorem}
\begin{proof}
	We see $(1)$ is obtained from \eqref{eq:betav}.
	We show $(2)$. Differentiating $\hat{\sigma}_{\boldsymbol{\nu}}(t)
	=S_{\boldsymbol{\nu}}(t,s(t))
	=\tilde{\gamma}(t) + s(t) \delta_{\boldsymbol{\nu}}(t)$, we see
	$\hat{\sigma}_{\boldsymbol{\nu}}'(t)=\tilde{\gamma}'
	+ s' \delta_{\boldsymbol{\nu}}
	+ s \delta_{\boldsymbol{\nu}}'$.
	By $\hat{\gamma}'(t)=l\boldsymbol{e}$, 
	\eqref{eq:betav}, \eqref{eq:stnu}
	and 
	$$
	s'(t)=\frac{1}{\beta_{\boldsymbol{\nu}}\sqrt{\kappa_3^2+\kappa_1^2}}
	\Big(
	l\beta_{\boldsymbol{\nu}}\big(2\kappa_1^2\kappa_1'+\kappa_3^2\kappa_1'+\kappa_1\kappa_3\kappa_3'\big)
	+\big(l'\beta_{\boldsymbol{\nu}}
	-l\beta_{\boldsymbol{\nu}}'\big)\kappa_1\big(\kappa_1^2+\kappa_3^2\big)	
	\Big),
	$$
	we have
	$$
	\hat{\sigma}_{\boldsymbol{\nu}}'(t)=
	\frac{\rho_{\boldsymbol{\nu}}}{\beta_{\boldsymbol{\nu}}^2}(\kappa_3\boldsymbol{e}
	+\kappa_1\boldsymbol{b}).
	$$
	Thus we obtain the result under the assumption
	$(\kappa_1,\kappa_3)\ne(0,0)$.
\end{proof}
We remark that the arguments for obtaining invariants $\beta_{\boldsymbol{\nu}}$ and $\rho_{\boldsymbol{\nu}}$ from a moving frame along a curve is based on \cite[Section 3]{izohflat}.
See \cite{folding1,folding2,ishi,istflat,istsw} for other studies of developable surfaces along a curve
on a surface or a frontal.
For cases where $S_{\boldsymbol{\nu}}$ is neither a cylinder nor a cone, we obtain the following results for the singularities of $S_{\boldsymbol{\nu}}$.
\begin{theorem}\label{thm:singsn}
	We assume that 
	$(\kappa_1, \kappa_3) \neq (0, 0)$ and
	$\beta_{\boldsymbol{\nu}} \neq 0$ at $t$.
	Then, the germ of $S_{\boldsymbol{\nu}}$ 
	at $(t,a)$ is a front for any $a$.
	Moreover, the germ $S_{\boldsymbol{\nu}}$ at $(t,s(t))$ 
	is a cuspidal edge if and only if
	$$
	\rho_{\boldsymbol{\nu}}\neq 0.
	$$
	The germ $S_{\boldsymbol{\nu}}$ at $(t,s(t))$ 
	is a swallowtail if and only if
	$$
	\rho_{\boldsymbol{\nu}}= 0,\quad
	\rho_{\boldsymbol{\nu}}'\neq 0.
	$$
\end{theorem}
\begin{proof}
	By Lemma \ref{lem:s2s3frl}, 
	$S_{\boldsymbol{\nu}}$ is a frontal with a unit normal vector
	$\boldsymbol{\nu}$. 
	Then noticing $(\kappa_1, \kappa_3) \neq (0, 0)$,
	we have 
	$$
	\rm rank
	\begin{pmatrix}
		(S_{\boldsymbol{\nu}})_t & \boldsymbol{\nu}_t \\
		(S_{\boldsymbol{\nu}})_a & \boldsymbol{\nu}_a 
	\end{pmatrix}
	= \rm rank
	\begin{pmatrix}
		(S_{\boldsymbol{\nu}})_t & -\kappa_1\boldsymbol{e} + \kappa_3\boldsymbol{b} \\
		\dfrac{\kappa_3\boldsymbol{e}+\kappa_1\boldsymbol{b}}{\kappa_3^2+\kappa_1^2} & 0
	\end{pmatrix}
	= 2
	$$
	which shows that $(S_{\boldsymbol{\nu}}, \boldsymbol{\nu})$ is an immersion.
	Therefore, $S_{\boldsymbol{\nu}}$ is a front. 
	From \eqref{eq:stnu} we can calculate $s(t)$,
	we know that $(t,s(t))$ is a singular point of $S_{\boldsymbol{\nu}}$.
	Then by \eqref{eq:s2ta}, the rank of
	$dS_{\boldsymbol{\nu}}|_{(t,s(t))}$ is one.
	By a direct calculation,
	the null vector field $\eta_{\boldsymbol{\nu}}$ 
	and the singularity identifier $\lambda_{\boldsymbol{\nu}}$ 
	for $S_{\boldsymbol{\nu}}$ are given as follows.
	$$
	\begin{cases}
		\eta_{\boldsymbol{\nu}}=
		\partial_t - \dfrac{l\kappa_3}{\sqrt{\kappa_3^2+\kappa_1^2}}\partial_a, \\
		\lambda_{\boldsymbol{\nu}}(t,a)=
		\det\big((S_{\boldsymbol{\nu}})_t, (S_{\boldsymbol{\nu}})_a, \boldsymbol{\nu}\big)=
		-\dfrac{l\kappa_1}{\sqrt{\kappa_3^2+\kappa_1^2}}
		+a\dfrac{\beta_{\boldsymbol{\nu}}}{\kappa_3^2+\kappa_1^2}.
	\end{cases}
	$$
	Calculating $\eta_{\boldsymbol{\nu}} \lambda_{\boldsymbol{\nu}}$
	and substituting $a=s(t)$,
	we obtain
	$$
	\eta_{\boldsymbol{\nu}} \lambda_{\boldsymbol{\nu}}|_{(t,s(t))}
	=
	-\dfrac{\rho_{\boldsymbol{\nu}}}
	{\beta_{\boldsymbol{\nu}}\sqrt{\kappa_3^2+\kappa_1^2}}.
	$$
	Then we have the assertion for the case of cuspidal edge.
	
	If $\kappa_1\ne0$, then $\rho_{\boldsymbol{\nu}}=0$ is equivalent to
	\begin{equation}\label{eq:lpnu}
		l'=
		-\dfrac{\beta_{\boldsymbol{\nu}}(\kappa_2\kappa_3+2\kappa_1')-\beta_{\boldsymbol{\nu}}'\kappa_1}{\kappa_1\beta_{\boldsymbol{\nu}}}l.
	\end{equation}
	Calculating $\eta_{\boldsymbol{\nu}}\eta_{\boldsymbol{\nu}}
	\lambda_{\boldsymbol{\nu}}$
	and substituting $a=s(t)$ and \eqref{eq:lpnu},
	we obtain
	$$
	\eta_{\boldsymbol{\nu}}\eta_{\boldsymbol{\nu}}
	\lambda_{\boldsymbol{\nu}}|_{(t,s(t))}
	=
	-\dfrac{\rho_{\boldsymbol{\nu}}'}{\beta_{\boldsymbol{\nu}}\sqrt{\kappa_3^2+\kappa_1^2}}
	$$
	under the condition \eqref{eq:lpnu}.
	Thus, we obtained the assertion that $S_{\boldsymbol{\nu}}$ is a swallowtail in the case of $\kappa_1\ne0$.
	If $\kappa_1=0$, then noticing
	$\beta_{\boldsymbol{\nu}}\kappa_3\ne0$, the condition
	$\rho_{\boldsymbol{\nu}}=0$ is equivalent to
	$$
	l(\kappa_2 \kappa_3 + 2\kappa_1')=0.
	$$
	Firstly, we consider the case of $l(t)=0$, 
	then we have
	$$
	\rho_{\boldsymbol{\nu}}'=
	\kappa_3 (\kappa_2 \kappa_3 + \kappa_1') (\kappa_2 \kappa_3 + 
	3 \kappa_1') l'
	=
	\beta_{\boldsymbol{\nu}} (\kappa_2 \kappa_3 + 
	3 \kappa_1') l'\quad
	\text{and}\quad
	\eta_{\boldsymbol{\nu}}\eta_{\boldsymbol{\nu}}\lambda_{\boldsymbol{\nu}}=
	-\dfrac{(\kappa_2 \kappa_3 + 3 \kappa_1') l'}{|\kappa_3|}.
	$$
	Secondly, we consider the case of $l(t)\ne0$.
	Then by $\rho_{\boldsymbol{\nu}}=0$, we have
	$\kappa_2 \kappa_3 + 2\kappa_1'=0$.
	If $\kappa_2=0$, then $\beta_{\boldsymbol{\nu}}=0$ holds.
	So we may assume that $\kappa_2\ne0$.
	We have
	$$
	\rho_{\boldsymbol{\nu}}'=
	\frac{\kappa_2\kappa_3^2}{4}\Big(\kappa_3 (4 l \kappa_2' - \kappa_2 l') + 
	6 l (\kappa_2 \kappa_3' + \kappa_1'')\Big),
	$$
	$$
	\eta_{\boldsymbol{\nu}}\eta_{\boldsymbol{\nu}}\lambda_{\boldsymbol{\nu}}
	=-\dfrac{1}{2|\kappa_3|}\Big(\kappa_3 (4 l \kappa_2' - \kappa_2 l') + 
	6 l (\kappa_2 \kappa_3' + \kappa_1'')\Big).
	$$
	Thus, we obtained the assertion that $S_{\boldsymbol{\nu}}$ is a swallowtail in the case of $\kappa_1=0$.
\end{proof}
Since we are interested in the case that $\hat{\gamma}(t)$ has a singular
point, we state the theorem in the case of $l(t)=0$.
In this case, $s(t)=0$.
\begin{corollary}
	Under the same assumption as in Theorem {\rm \ref{thm:singsn}},
	if $l(t)=0$, the following hold.
	The germ $S_{\boldsymbol{\nu}}$ at $(t,0)$ 
	is a cuspidal edge if and only if
	$$
	l'\kappa_1\ne0.
	$$
	The germ $S_{\boldsymbol{\nu}}$ at $(t,0)$ 
	is a swallowtail if and only if
	$$
	l'=0,\ \kappa_1l''\ne0\ \text{or}\
	\kappa_1=0,\  
	l'(\kappa_2\kappa_3+3\kappa_1')\ne0.
	$$
\end{corollary}

\subsection{Properties of the Surface $S_{\boldsymbol{b}}$}
In this subsection, we assume $(\kappa_2,\kappa_3)\neq(0,0)$ 
for any $t\in I$.
By a direct calculation, we have
\begin{equation}\label{eq:deltab}
	\delta_{\boldsymbol{b}}'(t)=\beta_{\boldsymbol{b}}\boldsymbol{w},\quad
	\left(
	\beta_{\boldsymbol{b}}(t)=\kappa_1\kappa_2^2
	+ \kappa_1\kappa_3^2 + \kappa_2\kappa_3' 
	- \kappa_2'\kappa_3,\quad
	\boldsymbol{w}(t)=
	\frac{\kappa_2\boldsymbol{e}+\kappa_3\boldsymbol{\nu}}{(\kappa_3^2+\kappa_2^2)^\frac{3}{2}}
	\right).
\end{equation}
If the developable surface $S_{\boldsymbol{b}}$ is non-cylindrical, then setting
\begin{equation}\label{eq:stb}
	s(t) 
	= -\frac{l\kappa_2\sqrt{\kappa_3^2+\kappa_2^2}}
	{\beta_{\boldsymbol{b}}},
\end{equation}
striction curve is obtained by $\hat{\sigma}_{\boldsymbol{b}}(t) = S_{\boldsymbol{b}}(t, s(t))$.
Under the assumption $(\kappa_2,\kappa_3)\neq(0,0)$.
The singular points of $S_{\boldsymbol{b}}$ 
satisfies that
$S(S_{\boldsymbol{b}}) = \{(t, a) \mid a = s(t)\}$.
We have the follows.
\begin{theorem}\label{thm:s2cylcone}
	$(1).$The developable surface $S_{\boldsymbol{b}}$ is a cylinder if and only if
	$$
	\beta_{\boldsymbol{b}}\equiv 0,
	$$
	where $\equiv$ stands for the equality holds identically.
	Similarly, $S_{\boldsymbol{b}}$ is non-cylindrical if 
	$\beta_{\boldsymbol{b}}$ never vanishes on $I$.
	
	$(2).$The developable surface $S_{\boldsymbol{b}}$ is a {\it cone} if and only if
	$\beta_{\boldsymbol{b}} \neq 0$ and $\rho_{\boldsymbol{b}}\equiv0$, where
	$$
	\rho_{\boldsymbol{b}}(t)
	=
	l\big(\beta_{\boldsymbol{b}}(\kappa_1\kappa_3-2\kappa_2')+\beta_{\boldsymbol{b}}'\kappa_2\big)-l'\kappa_2\beta_{\boldsymbol{b}}.
	$$
\end{theorem}
\begin{proof}
	We see $(1)$ is obtained from \eqref{eq:deltab}.
	We show $(2)$. Differentiating $\hat{\sigma}_{\boldsymbol{b}}(t)
	=S_{\boldsymbol{b}}(t,s(t))=\tilde{\gamma}(t) + s(t)\delta_{\boldsymbol{b}}(t)$, we see
	$\hat{\sigma}_{\boldsymbol{b}}'(t)=\tilde{\gamma}'
	+ s'\delta_{\boldsymbol{b}}
	+ s\delta_{\boldsymbol{b}}'$.
	By $\hat{\gamma}'(t)=l\boldsymbol{e}$, \eqref{eq:deltab}, \eqref{eq:stb} and 
	$$
	s'(t)=-\frac{1}{\beta_{\boldsymbol{\nu}}\sqrt{\kappa_3^2+\kappa_2^2}}
	\Big(
	l\beta_{\boldsymbol{b}}\big(2\kappa_2^2\kappa_2'+\kappa_3^2\kappa_2'+\kappa_2\kappa_3\kappa_3'\big)
	+\big(l'\beta_{\boldsymbol{b}}
	-l\beta_{\boldsymbol{b}}'\big)\kappa_2\big(\kappa_2^2+\kappa_3^2\big)	
	\Big),
	$$
	we have
	$$
	\hat{\sigma}_{\boldsymbol{b}}'(t)=
	\frac{\rho_{\boldsymbol{b}}}{\beta_{\boldsymbol{b}}^2}(\kappa_3\boldsymbol{e}
	-\kappa_2\boldsymbol{\nu}).
	$$
	Thus we obtain the result under the assumption
	$(\kappa_2,\kappa_3)\ne(0,0)$.
\end{proof}
We remark that 
the arguments for obtaining invariants $\beta_{\boldsymbol{b}}$
and $\rho_{\boldsymbol{b}}$ 
from a moving frame along a curve
is based on \cite[Section 3]{izohflat}.
See \cite{folding1,folding2,ishi,istflat,istsw} for other studies of developable surfaces along a curve
on a surface or a frontal.
For cases where $S_{\boldsymbol{b}}$ is neither a cylinder nor a cone, 
we obtain the following results for the singularities of $S_{\boldsymbol{b}}$.
\begin{theorem}\label{thm:sbceswa}
	We assume that $(\kappa_2, \kappa_3)\neq0$, 
	$\beta_{\boldsymbol{b}}\neq0$ at $t$.
	Then, the germ of $S_{\boldsymbol{b}}$ 
	at $(t,a)$ is a front at any $a$.
	Moreover, the germ $S_{\boldsymbol{b}}$ at $(t,s(t))$ 
	is a cuspidal edge if and only if
	$$
	\rho_{\boldsymbol{b}}\neq 0.
	$$
	The germ $S_{\boldsymbol{b}}$ at $(t,s(t))$ 
	is a swallowtail if and only if
	$$
	\rho_{\boldsymbol{b}}= 0,\quad
	\rho_{\boldsymbol{b}}'\neq 0.
	$$
\end{theorem}
\begin{proof}
	From Lemma \ref{lem:s2s3frl}, 
	$S_{\boldsymbol{b}}$ is a frontal 
	with a unit normal vector $\boldsymbol{b}$.
	The noticing $(\kappa_2, \kappa_3) \neq (0, 0)$, 
	we have
	$$
	\rm rank
	\begin{pmatrix}
		(S_{\boldsymbol{b}})_t & \boldsymbol{b}_t \\
		(S_{\boldsymbol{b}})_a & \boldsymbol{b}_a 
	\end{pmatrix}
	= \rm rank
	\begin{pmatrix}
		(S_{\boldsymbol{b}})_t & -\kappa_2\boldsymbol{e} - \kappa_3\boldsymbol{\nu} \\
		\dfrac{\kappa_3\boldsymbol{e}-\kappa_2\boldsymbol{\nu}}{\kappa_3^2+\kappa_2^2} & 0
	\end{pmatrix}
	= 2
	$$
	which shows that $(S_{\boldsymbol{b}}, \boldsymbol{b})$ is an immersion.
	Therefore, $S_{\boldsymbol{b}}$ is a front. 
	From \eqref{eq:stb} we can calculate $s(t)$,
	we know that $(t,s(t))$ is a singular point of $S_{\boldsymbol{b}}$.
	Then by \eqref{eq:s3ta}, the rank of
	$dS_{\boldsymbol{b}}|_{(t,s(t))}$ is one.
	By a direct calculation,
	the null vector field $\eta_{\boldsymbol{b}}$ 
	and the singularity identifier $\lambda_{\boldsymbol{b}}$ 
	for $S_{\boldsymbol{b}}$ are given as follows.
	$$
	\begin{cases}
		\eta_{\boldsymbol{b}}=\partial_t - \Big(\dfrac{l\kappa_3}{\sqrt{\kappa_3^2+\kappa_2^2}}\Big)\partial_a \\
		\lambda_{\boldsymbol{b}}(t,a)=\det\big((S_{\boldsymbol{b}})_t, (S_{\boldsymbol{b}})_a, \boldsymbol{b}\big)=
		-\dfrac{l\kappa_2}{\sqrt{\kappa_3^2+\kappa_2^2}}
		-a\dfrac{\beta_{\boldsymbol{b}}}{\kappa_3^2+\kappa_2^2}
	\end{cases}
	$$
	Calculating $\eta_{\boldsymbol{b}} \lambda_{\boldsymbol{b}}$
	and substituting $a=s(t)$,
	we obtain
	$$
	\eta_{\boldsymbol{b}} \lambda_{\boldsymbol{b}}|_{(t,s(t))}
	=
	\dfrac{\rho_{\boldsymbol{b}}}
	{\beta_{\boldsymbol{b}}\sqrt{\kappa_3^2+\kappa_2^2}}.
	$$
	Then we have the assertion for the case of cuspidal edge.
	
	If $\kappa_2\ne0$, then $\rho_{\boldsymbol{b}}=0$ is equivalent to
	\begin{equation}\label{eq:lpb}
		l'=
		\dfrac{\beta_{\boldsymbol{b}}(\kappa_1\kappa_3-2\kappa_2')+\beta_{\boldsymbol{b}}'\kappa_2}{\kappa_2\beta_{\boldsymbol{b}}}l.
	\end{equation}
	Calculating $\eta_{\boldsymbol{b}}\eta_{\boldsymbol{b}}
	\lambda_{\boldsymbol{b}}$
	and substituting $a=s(t)$ and \eqref{eq:lpb},
	we obtain
	$$
	\eta_{\boldsymbol{b}}\eta_{\boldsymbol{b}}
	\lambda_{\boldsymbol{b}}|_{(t,s(t))}
	=
	\dfrac{\rho_{\boldsymbol{b}}'}{\beta_{\boldsymbol{b}}\sqrt{\kappa_3^2+\kappa_2^2}}
	$$
	under the condition \eqref{eq:lpb}.
	Thus, we obtained the assertion that $S_{\boldsymbol{b}}$ is a swallowtail in the case of $\kappa_2\ne0$.
	If $\kappa_2=0$, then noticing
	$\beta_{\boldsymbol{b}}\kappa_3\ne0$, the condition
	$\rho_{\boldsymbol{b}}=0$ is equivalent to
	$$
	l(\kappa_1 \kappa_3 - 2\kappa_2')=0.
	$$
	Firstly, we consider the case of $l(t)=0$, 
	then we have
	$$
	\rho_{\boldsymbol{b}}'=
	l'\kappa_3(\kappa_1 \kappa_3 - \kappa_2')(\kappa_1\kappa_3-3\kappa_2')
	=
	\beta_{\boldsymbol{b}}(\kappa_1\kappa_3-3\kappa_2')l'\quad
	\text{and}\quad
	\eta_{\boldsymbol{b}}\eta_{\boldsymbol{b}}\lambda_{\boldsymbol{b}}
	=\dfrac{(\kappa_1\kappa_3-3\kappa_2')l'}{|\kappa_3|}.
	$$
	Secondly, we consider the case of $l(t)\ne0$.
	Then by $\rho_{\boldsymbol{b}}=0$, we have
	$\kappa_1 \kappa_3 - 2\kappa_2'=0$.
	If $\kappa_1=0$, then $\beta_{\boldsymbol{b}}=0$ holds.
	So we may assume that $\kappa_1\ne0$.
	We have
	$$
	\rho_{\boldsymbol{b}}'=
	\dfrac{\kappa_1\kappa_3^2}{4} \Big(\kappa_3 (4 l \kappa_1' - \kappa_1 l') + 
	6 l (\kappa_1 \kappa_3' - \kappa_2'')\Big),
	$$
	$$
	\eta_{\boldsymbol{b}}\eta_{\boldsymbol{b}}\lambda_{\boldsymbol{b}}
	=\dfrac{1}{2|\kappa_3|}\Big(\kappa_3 (4 l \kappa_1' - \kappa_1 l') + 
	6 l (\kappa_1 \kappa_3' - \kappa_2'')\Big).
	$$
	Thus, we obtained the assertion that $S_{\boldsymbol{b}}$ is a swallowtail in the case of $\kappa_2=0$.
\end{proof}
Since we are interested in the case of $\hat{\gamma}(t)$ has a singular
point, we state the theorem in the case of $l(u)=0$.
In this case, $s(t)=0$.
\begin{corollary}
	Under the same assumption in Theorem {\rm \ref{thm:sbceswa}},
	if $l(t)=0$, the following hold.
	The germ $S_{\boldsymbol{b}}$ at $(t,0)$ 
	is a cuspidal edge if and only if
	$$
	l'\kappa_2\ne0.
	$$
	The germ $S_{\boldsymbol{b}}$ at $(t,0)$ 
	is a swallowtail if and only if
	$$
	l'=0,\ \kappa_2l''\ne0\
	\text{or}\
	\kappa_2=0,\ 
	l'(\kappa_1\kappa_3-3\kappa_2')\ne0.
	$$
\end{corollary}

\section{Application gluing of two surfaces}
\subsection{Gluing of two surfaces}\label{sec:gluing}
In this section, we study the gluing of two frontal surfaces $f_1$ and $f_2$ along a curve $\hat\gamma$.
Since $f_i$ $(i=1,2)$ are frontals, there are unit normal vectors $\boldsymbol{\nu}_i$.
So we can construct developable surfaces $S_{\boldsymbol{\nu}_1}, S_{\boldsymbol{b}_1}$ and $S_{\boldsymbol{\nu}_2}, S_{\boldsymbol{b}_2}$. 
Looking at geometries of these surfaces, we study the geometry of gluing of two surfaces along the gluing locus $\hat{\gamma}$. 
We give conditions that the developable surfaces $S_{\boldsymbol{\nu}_i}$ and $S_{\boldsymbol{b}_i}$ 
are cylindrical, conical and having cuspidal edge or
swallowtail singularities.
Furthermore, we study how the angle between two normal vectors
of $f_1$ and $f_2$ affects the gluing properties.

\begin{definition}
	Let $U\subset \R^2$ be an open neighborhood of the origin.
	We set $U_1=U\cap \{(t,a)\in\R^2\,|\,a\geq0\}$
	and
	$U_2=U\cap \{(t,a)\in\R^2\,|\,a\leq0\}$.
	Let $f_i: U_i \to \R^3$ be two fronts $(i=1,2)$
	satisfying 
	$$
	f_1|_I=f_2|_I,
	$$
	where $I=U\cap \{(t,0)\in\R^2\}$.
	Let $\hat{\alpha}_i (t)= f_i|_I(t)=f_i(t,0)$ and let us set
	$$
	\hat\gamma(t)=\hat\alpha_1(t)=\hat\alpha_2(t).
	$$
	Then the triple $(f_1, f_2, \hat{\gamma})$ is called a {\it glue}
	of $f_1$ and $f_2$ along $\hat\gamma$.
\end{definition}
In the above definition, since $f_1|_I=f_2|_I$,
one can interpret that the two surfaces are glued along $\hat\gamma$.
The curve $\hat{\gamma}$ is called a {\it gluing locus}.

We assume that there exist a function $l(t)$ and a unit vector $\boldsymbol{e}$ 
such that $\hat{\gamma}' = l \boldsymbol{e}$. 
Let $\boldsymbol{\nu}_i$ be the unit normal vector of $f_i$ and let us set $\boldsymbol{b}_i=\boldsymbol{e}\times \boldsymbol{\nu}_i$
for $i=1,2$. 
Then we have two frames 
$$
\{\boldsymbol{e}, \boldsymbol{\nu}_i, \boldsymbol{b}_i\}\quad
(i=1,2)
$$
along $\hat{\gamma}$.
For these frames, 
the functions $\kappa_{i1}(t)$, $\kappa_{i2}(t)$ and $\kappa_{i3}(t)$ 
are determined by the Frenet-Serret type formula \eqref{eq:frenet}, they are regard as invariants of $f_i$.
Moreover, let $\kappa_{\boldsymbol{\nu}_i},\kappa_{g_i},\tau_{g_i}$ denote the 
{\it normal curvature}, {\it geodesic curvature} with respect to
$\boldsymbol{\nu}_i$ and
{\it geodesic torsion} of $\hat{\gamma}$ as a curve on $f_i$,
which are given by
$$
\kappa_{\boldsymbol{\nu}_i}=\dfrac{\hat{\gamma}''\cdot\boldsymbol{\nu}_i}{|\hat{\gamma}'|^2},\
\kappa_{g_i}=\dfrac{\det(\hat{\gamma}',\hat{\gamma}'',\boldsymbol{\nu}_i)}{|\hat{\gamma}'|^3},\ 
\tau_{g_i}=\dfrac{\det(\hat{\gamma}',\boldsymbol{\nu}_i,\boldsymbol{\nu}_i')}{|\hat{\gamma}'|^2}.
$$
Let $\theta$ be the angle between 
$\boldsymbol{\nu}_1$ and $\boldsymbol{\nu}_2$
, and it can be a function of parameter $t$.
Then we have the following lemma.
\begin{lemma}\label{lm:k2-k1}
	Under the above settings, it holds that
	\begin{equation}\label{eq:kttheta}
		\begin{cases}
			\kappa_{i1} =l\kappa_{\boldsymbol{\nu}_i}\\
			\kappa_{i2} =-|l|\kappa_{g_i}\\
			\kappa_{i3} =l\tau_{g_i}
		\end{cases}
	\end{equation}
	and
	\begin{equation}\label{eq:k123theta}
		\begin{cases}
			\kappa_{21} =\kappa_{11}\cos\theta+\kappa_{12}\sin\theta \\
			\kappa_{22} =-\kappa_{11}\sin\theta+\kappa_{12}\cos\theta \\
			\kappa_{23} =\kappa_{13}+\theta'.
		\end{cases}
	\end{equation}
\end{lemma}
\begin{proof}
	Since $\hat{\gamma}'=l\boldsymbol{e}$,
	we have $\hat{\gamma}''=l'\boldsymbol{e}+l(\kappa_{i1}
	\boldsymbol{v}_i+\kappa_{i2}\boldsymbol{b}_i)$,
	and by \eqref{eq:frenet}, we have
	$\boldsymbol{\nu}_i'=-\kappa_{i1}
	\boldsymbol{e}_i+\kappa_{i3}\boldsymbol{b}_i$.
	Calculating them, we obtained
	$$
	\kappa_{\boldsymbol{\nu}_i}
	=\dfrac{l\kappa_{i1}}{|l|^2},\
	\kappa_{g_i}=-\dfrac{l^2\kappa_{i2}}{|l|^3},\ 
	\tau_{g_i}=\dfrac{l\tau_{i3}}{|l|^2}.
	$$
	Thus, we obtained the result of \eqref{eq:kttheta}.
	
	The vectors $\boldsymbol{\nu}_2, \boldsymbol{b}_2$ of the frame $\{\boldsymbol{e}, \boldsymbol{\nu}_2, \boldsymbol{b}_2\}$ 
	is obtained by rotating $\boldsymbol{\nu}_1, \boldsymbol{b}_1$
	of the frame 
	$\{\boldsymbol{e}, \boldsymbol{\nu}_1, \boldsymbol{b}_1\}$
	around $\boldsymbol{e}$ by angle $\theta$ respectively.
	Then by using Rodrigues' rotation formula, 
	we obtained 
	\begin{equation}\label{eq:rodrigues}
		\begin{array}{rl}
			\boldsymbol{\nu}_2&=\cos\theta\boldsymbol{\nu}_1+\sin\theta\boldsymbol{b}_1,\\
			\boldsymbol{b}_2&=\cos\theta\boldsymbol{b}_1-\sin\theta\boldsymbol{\nu}_1.
		\end{array}
	\end{equation}
	By the Frenet-Serret type formula \eqref{eq:frenet}, we can get
	$$
	\kappa_{i1} =\boldsymbol{e}'\cdot\boldsymbol{\nu}_i,\
	\kappa_{i2} =\boldsymbol{e}'\cdot\boldsymbol{b}_i,\
	\kappa_{i3} =\boldsymbol{\nu}_i'\cdot\boldsymbol{b}_i.
	$$
	Therefore, we obtained the result of \eqref{eq:k123theta}.
\end{proof}
As in Section \ref{sec:dev}, we construct four developable
surfaces by using $\boldsymbol{\nu}_i, \boldsymbol{b}_i$ $(i=1,2)$.
Let us set these surfaces
$$
S_{\boldsymbol{\nu}_1},\quad
S_{\boldsymbol{b}_1},\quad
S_{\boldsymbol{\nu}_2},\quad
S_{\boldsymbol{b}_2}
$$
respectively.
It should be noted that the surfaces $S_{\boldsymbol{\nu}_1}$ and $S_{\boldsymbol{b}_1}$ are obtained just from the information of $f_1$ itself without gluing.
However, considering the gluing $(f_2,f_1,\hat\gamma)$,
the surfaces $S_{\boldsymbol{\nu}_1}$ and $S_{\boldsymbol{b}_1}$ is regarded as the surfaces along $\hat\gamma$ on $f_2$, the author believes it will be meaningful.
Since $\theta$ is the angle between $\boldsymbol{\nu}_1$ and $\boldsymbol{\nu}_2$, the surfaces $S_{\boldsymbol{\nu}_2}, S_{\boldsymbol{b}_2}$ are obtained by rotating the each ruling by $\theta$ from $S_{\boldsymbol{\nu}_1},S_{\boldsymbol{b}_1}$ 
along $\hat\gamma$ respectively.
Furthermore, $S_{\boldsymbol{b}_1}$ is obtained by rotating the each ruling by $\pi/2$ from $S_{\boldsymbol{\nu}_1}$ along $\hat\gamma$, these four surfaces are not independent.
However, we treat them separately since the conditions are different.
We define special gluing as when these developable surfaces are special.
\begin{definition}\label{def:cylcone}
	The gluing $(f_1, f_2, \hat{\gamma})$ is said to be
	\begin{itemize}
		\item {\it $S_{\boldsymbol{\nu}_i}$-cylindrical} 
		if $S_{\boldsymbol{\nu}_i}$ is a cylinder.
		\item {\it $S_{\boldsymbol{b}_i}$-cylindrical} 
		if $S_{\boldsymbol{b}_i}$ is a cylinder.
		\item {\it $S_{\boldsymbol{\nu}_i}$-conical} 
		if $S_{\boldsymbol{\nu}_i}$ is a cone.
		\item {\it $S_{\boldsymbol{b}_i}$-conical} 
		if $S_{\boldsymbol{b}_i}$ is a cone.
	\end{itemize}
\end{definition}
Furthermore, we define special gluing at a point when these developable surfaces have a fundamental singularity.
We remark that each ruling has a unique singular point for a developable surface.
\begin{definition}\label{def:glucesw}
	The gluing $(f_1, f_2, \hat{\gamma})$ at $t_0$ is said to be
	\begin{itemize}
		\item {\it $S_{\boldsymbol{\nu}_i}$-cuspidal edgy}
		if $S_{\boldsymbol{\nu}_i}$ is a cuspidal edge at $(t_0,a)$ for some $a\in\R$.
		\item {\it $S_{\boldsymbol{\nu}_i}$-swallowtailed}
		if $S_{\boldsymbol{\nu}_i}$ is a swallowtail at $(t_0,a)$ for some $a\in\R$.
		\item {\it $S_{\boldsymbol{b}_i}$-cuspidal edgy} if $S_{\boldsymbol{b}_i}$ is a cuspidal edge at $(t_0,a)$ for some $a\in\R$.
		\item {\it $S_{\boldsymbol{b}_i}$-swallowtailed}
		if $S_{\boldsymbol{b}_i}$ is a swallowtail at $(t_0,a)$ for some $a\in\R$.
	\end{itemize}
\end{definition}

\subsection{Glue with cylinder or cone}
We give the conditions of the special gluings given in
Definition \ref{def:cylcone} and Definition \ref{def:glucesw} in terms of the invariants of original surfaces.
Since we are interested in the case where the curve $\hat{\gamma}$ has a singularity.
By \eqref{eq:k123theta}, we can obtain the conditions for $S_{\boldsymbol{\nu}_2}$ to be a cylinder or a cone in terms of the invariant of $S_{\boldsymbol{\nu}_1}$. 
Let $(\kappa_{i1}, \kappa_{i2}, \kappa_{i3})$ be the invariant of the frame $\{\boldsymbol{e},\boldsymbol{\nu}_i,\boldsymbol{b}_i\}$ 
as in Section \ref{sec:gluing}.
Furthermore, we set
\begin{align*}
	\beta_{\boldsymbol{\nu}_i}(t)&=\kappa_{i1}^2\kappa_{i2}
	+ \kappa_{i2}\kappa_{i3}^2 
	+ \kappa_{i1}'\kappa_{i3}  - \kappa_{i1}\kappa_{i3}',\\
	\rho_{\boldsymbol{\nu}_i}(t)
	&=
	l\big(\beta_{\boldsymbol{\nu}_i}(\kappa_{i2}\kappa_{i3}+2\kappa_{i1}')-\beta_{\boldsymbol{\nu}_i}'\kappa_{i1}\big)+l'\kappa_{i1}\beta_{\boldsymbol{\nu}_i}.
\end{align*}

\begin{theorem}\label{thm:4.5}
	The developable surface $S_{\boldsymbol{\nu}_i}$ is
	a cylinder if and only if $\beta_{\nu_i}\equiv0$.
	The developable surface $S_{\boldsymbol{\nu}_i}$ is a
	cone if and only if $\beta_{\nu_i}\neq0$ and $\rho_{\nu_i}\equiv0$.
\end{theorem}
\begin{proof}
	In Section \ref{sub:sv}, we obtained the conditions $\rho_{\boldsymbol{\nu}}$
	and $\beta_{\boldsymbol{\nu}}$ for $S_{\boldsymbol{\nu}}$ 
	in a general case expressed by the invariants $(\kappa_{1},\kappa_{2},\kappa_{3})$.
	By substituting $(\kappa_{i1},\kappa_{i2},\kappa_{i3})$, 
	we can get $\beta_{\boldsymbol{\nu}_i}$ and $\rho_{\boldsymbol{\nu}_i}$ about $S_{\boldsymbol{\nu}_i}$. 
	By \eqref{eq:k123theta}, we can get $\beta_{\boldsymbol{\nu}_2}$ and $\rho_{\boldsymbol{\nu}_2}$ represented by $(\kappa_{11},\kappa_{12},\kappa_{13})$ is
	\begin{align*}
		\beta_{\boldsymbol{\nu}_2}(t)=&
		(\cos\theta \kappa_{12} - \kappa_{11} \sin\theta) (\cos\theta \kappa_{11} + 
		\kappa_{12} \sin\theta)^2 \\
		&+(\kappa_{13} + 
		\theta') \Big((\cos\theta \kappa_{12}  - 
		\kappa_{11} \sin\theta )(\kappa_{13} + 2 \theta') + 
		\cos\theta \kappa_{11}' + 
		\sin\theta \kappa_{12}'\Big) \\
		&- (\cos\theta \kappa_{11} + 
		\kappa_{12} \sin\theta) (\kappa_{13}' + \theta''),\\
		\rho_{\boldsymbol{\nu}_2}(t)=&l\Big(\beta_{\boldsymbol{\nu}_2} (\cos\theta \kappa_{12} - \kappa_{11} \sin\theta) (\kappa_{13} + 
		\theta') \\
		&+ 2 \beta_{\boldsymbol{\nu}_2} \big(\theta'(\cos\theta \kappa_{12} - \kappa_{11} \sin\theta) + 
		\cos\theta \kappa_{11}' + \sin\theta \kappa_{12}'\big)\\
		&-\beta_{\boldsymbol{\nu}_2}'(\cos\theta \kappa_{11} + \kappa_{12} \sin\theta)\Big)
		+l'\beta_{\boldsymbol{\nu}_2} \Big(\cos\theta \kappa_{11} + \kappa_{12} \sin\theta\Big).
	\end{align*}
	This illustrates the relationship between $S_{\boldsymbol{\nu}_1}$ and $S_{\boldsymbol{\nu}_2}$.
\end{proof}

When the angle $\theta$ between $\boldsymbol{\nu}_1$ and $\boldsymbol{\nu}_2$ 
is a special value, we get the following corollaries.

\begin{corollary}
	Let $\theta=k\pi/2$, $k$ is an integer.
	Then the following hold.
	
	The developable surface $S_{\boldsymbol{\nu}_2}$ is a cylinder
	if and only if 
	the developable surface $S_{\boldsymbol{b}_1}$ is a cylinder.
	
	The developable surface $S_{\boldsymbol{\nu}_2}$ is a cone
	if and only if 
	the developable surface $S_{\boldsymbol{b}_1}$ is a cone.
\end{corollary}
\begin{corollary}
	Let $\theta=k\pi$, $k$ is an integer.
	Then the following hold.
	
	The developable surface $S_{\boldsymbol{\nu}_2}$ is a cylinder
	if and only if 
	the developable surface $S_{\boldsymbol{\nu}_1}$ is a cylinder.
	
	The developable surface $S_{\boldsymbol{\nu}_2}$ is a cone
	if and only if 
	the developable surface $S_{\boldsymbol{\nu}_1}$ is a cone.
\end{corollary}

\subsection{Singularities of glue with tangent surfaces}
Since the frame $\{\boldsymbol{e}, \boldsymbol{\nu}_2, \boldsymbol{b}_2\}$ 
is obtained by rotating the frame 
$\{\boldsymbol{e}, \boldsymbol{\nu}_1, \boldsymbol{b}_1\}$ around 
$\boldsymbol{e}$ by angle $\theta$, as in \eqref{eq:rodrigues},
the conditions for $S_{\boldsymbol{\nu}_2}$ should be
expressed by the invariant of $S_{\boldsymbol{\nu}_1}$ and the rotation angle $\theta$.
Let $(\kappa_{i1}, \kappa_{i2}, \kappa_{i3})$ be the invariant of $S_{\boldsymbol{\nu}_i}$.
\begin{theorem}
	We assume that $(\kappa_{i1},\kappa_{i3})\neq(0,0)$ and $\beta_{\boldsymbol{\nu}_i}\neq0$ at $t$. 
	Then, the germ of $S_{\boldsymbol{\nu}_i}$ at $(t,a)$ is a front
	for any $a$. 
	Moreover, the germ $S_{\boldsymbol{\nu}_i}$ at $(t, s(t))$ is a cuspidal edge if and only if  
	$$
	\beta_{\boldsymbol{\nu}_i}\neq0,\ 
	\rho_{\boldsymbol{\nu}_i}\neq0.
	$$
	The germ $S_{\boldsymbol{\nu}_i}$ at $(t, s(t))$ is a swallowtail if and only if 
	$$
	\beta_{\boldsymbol{\nu}_i}\neq0,\ 
	\rho_{\boldsymbol{\nu}_i}=0,\ 
	\rho_{\boldsymbol{\nu}_i}'\neq0.
	$$
\end{theorem}

Through \eqref{eq:k123theta}, we use the invariant $(\kappa_{11}, \kappa_{12}, \kappa_{13})$ of $S_{\boldsymbol{\nu}_1}$ to express the conditions of $S_{\boldsymbol{\nu}_2}$ in cases of singularities, and obtain the following result.

\begin{corollary}\label{coro:v2}
	
	If $l(t)=0$, the following hold.\\
	The germ $S_{\boldsymbol{\nu}_2}$ at $(t,0)$ 
	is a cuspidal edge if and only if
	$$
	l'(\kappa_{11}\cos\theta + \kappa_{12} \sin\theta)\ne0.
	$$
	The germ $S_{\boldsymbol{\nu}_2}$ at $(t,0)$ 
	is a swallowtail if and only if
	$$l'=0,\ l''(\kappa_{11}\cos\theta+\kappa_{12}\sin\theta)\ne0\ \text{or}
	$$
	$$\kappa_{11}\cos\theta+\kappa_{12}\sin\theta=0
	,\ l'\big((\kappa_{12}\cos\theta-\kappa_{11}\sin\theta)(\kappa_{13}+4\theta')+3(\kappa_{11}'\cos\theta+\kappa_{12}'\sin\theta)\big)\ne0.
	$$
\end{corollary}
\section{Examples of $S_{\boldsymbol{\nu}_i}$ in special case}
In this section, we give several examples
which appeared in this paper.

\begin{example}
	We give an example of
	$S_{\boldsymbol{\nu}_2}$-cylindrical glue 
	and it is obtained by rotating the unit normal vector of the wave surface. 
	Let us set
	$\hat{\gamma}(u,0)=(\cos u,\sin u,u)$ and let us set $f_1$ and $f_2$ by
	$$
	f_1(u,v)=\hat{\gamma}(u)+v\big(0,1,0\big),\
	f_2(u,v)=\hat{\gamma}(u)+v\big(0,0,1\big)
	$$
	where the gluing locus is $\hat{\gamma}$.
	They are shown in Figure \ref{fig:ex1}. Let us set
	$$
	\boldsymbol{e}=\frac{1}{\sqrt{2}}\big(-\sin u,\cos u,1\big),
	$$
	$$
	\boldsymbol{\nu}_1=-\frac{1}{\sqrt{1+\sin^2u}}\big(1,0,\sin u\big),
	$$
	$$
	\boldsymbol{\nu}_2=\big(\cos u,\sin u,0\big).
	$$
	Then $\hat{\gamma}'=\sqrt{2}\boldsymbol{e}$,
	where the length function of $\hat{\gamma}$ is $l(u)=\sqrt{2}$.
	The vectors $\boldsymbol{\nu}_1$, $\boldsymbol{\nu}_2$ are the unit normal vectors of $f_1$ and $f_2$ respectively.
	We set $\boldsymbol{b}_1=
	\boldsymbol{e}\times\boldsymbol{\nu}_1$ and $\boldsymbol{b}_2=
	\boldsymbol{e}\times\boldsymbol{\nu}_2$.
	Then by the Frenet-Serre type formula, $(\kappa_{11},\kappa_{12},\kappa_{13})$ and $(\kappa_{21},\kappa_{22},\kappa_{23})$ are
	$$
	\kappa_{11}=\frac{\cos u}{\sqrt{3-\cos2u}},\
	\kappa_{12}=\frac{\sin u}{\sqrt{1+\sin^2u}},\
	\kappa_{13}=\frac{\sqrt{2}\cos^2u}{\cos2u-3};
	$$
	$$
	\kappa_{21}=-\frac{1}{\sqrt{2}},\ 
	\kappa_{22}=0,\ 
	\kappa_{23}=\frac{1}{\sqrt{2}}.
	$$
	Set 
	$$
	S_{\boldsymbol{\nu}_2}(u,v)=
	\hat{\gamma}+v\Big(\dfrac{\kappa_{23}\boldsymbol{e}+\kappa_{21}\boldsymbol{b}_2}{\sqrt{\kappa_{23}^2+\kappa_{21}^2}}\Big).
	$$
	Then $S_{\boldsymbol{\nu}_2}$ and $f_2$ have the same image and unit normal vector.
	Moreover, they are glued along $\hat{\gamma}$ and $f_1$, as shown in the Figure \ref{fig:ex1}.
	Let the function $\theta(u)$ be the angle between $\boldsymbol{\nu}_1$ and $\boldsymbol{\nu}_2$.
	We set
	$$
	\begin{cases}
		\sin\theta=-\frac{\sqrt{1-\cos2u}}{\sqrt{1+\sin^2u}},\ 
		\cos\theta=-\frac{\cos u}{\sqrt{1+\sin^2u}},\ 
		\theta'=-\frac{2\sqrt{2}}{\cos2u-3}\ 
		\text{when}\ \sin u\geqslant0;\\
		\sin\theta=\frac{\sqrt{1-\cos2u}}{\sqrt{1+\sin^2u}},\
		\cos\theta=-\frac{\cos u}{\sqrt{1+\sin^2u}},\ 
		\theta'=\frac{2\sqrt{2}}{\cos2u-3}\ 
		\text{when}\ \sin u<0.
	\end{cases}
	$$
	Then by Lemma \ref{lm:k2-k1}, we know that by rotating the unit normal vector of $S_{\boldsymbol{\nu}_1}$ around $\boldsymbol{e}$, we get $S_{\boldsymbol{\nu}_2}$.
	Then calculating $\beta_{\boldsymbol{\nu}_2}$,
	we obtain $\beta_{\boldsymbol{\nu}_2}=0$.
	This shows that $S_{\boldsymbol{\nu}_2}$ is cylinder
	and $(f_1,f_2,\hat{\gamma})$ is $S_{\boldsymbol{\nu}_2}$-cylindrical glue.
\end{example}
\begin{figure}[htbp]
	\centering
	\includegraphics[width=8cm]{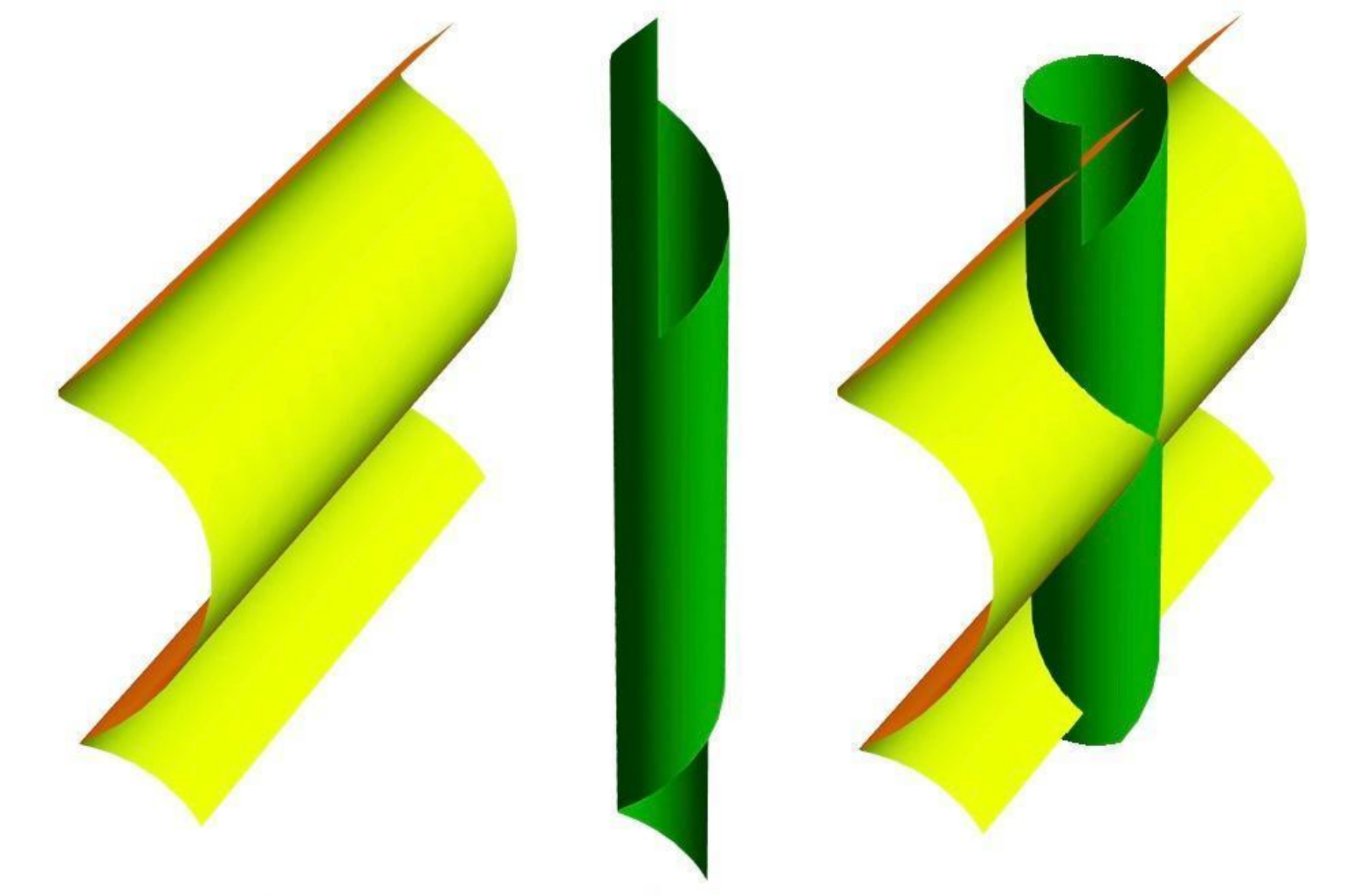}
	\caption{
		From left to right,
		image of $(f_1, S_{\boldsymbol{\nu}_1})$,
		image of $(f_2, S_{\boldsymbol{\nu}_2})$, image of Glued ($S_{\boldsymbol{\nu}_1}, S_{\boldsymbol{\nu}_2}$).}
	\label{fig:ex1}
\end{figure}

\begin{example}
	We give an example of
	$S_{\boldsymbol{\nu}_1}$-cylindrical glue 
	and also $S_{\boldsymbol{\nu}_2}$-conical glue. 
	Let us set 
	$\hat{\gamma}(u,0)=(\cos u,\sin u,1)$,
	and let us set $f_1$ and $f_2$ by
	$$
	f_1(u,v)=\big(\cos u,\sin u,v\big),\
	f_2(u,v)=\big(v\cos u,v\sin u,v\big)
	$$
	where the gluing locus is $\hat{\gamma}$.
	They are shown in the Figure \ref{fig:ex5}.
	Let us set
	$$
	\boldsymbol{e}=\big(-\sin u,\cos u,0\big),
	$$
	$$
	\boldsymbol{\nu}_1=\big(\cos u,\sin u,1\big),
	$$
	$$
	\boldsymbol{\nu}_2=\frac{1}{\sqrt{2}}\big(\cos u.\sin u,-1\big).
	$$
	Then $\hat{\gamma}'=\boldsymbol{e}$,
	where the length function of $\hat{\gamma}$ is $l(u)=1$.
	The vectors $\boldsymbol{\nu}_1$, $\boldsymbol{\nu}_2$ are the unit normal vectors of $f_1$ and $f_2$ respectively.
	Moreover, the angle between $\boldsymbol{\nu}_1$
	and $\boldsymbol{\nu}_2$ is $\frac{\pi}{4}$.
	We set $\boldsymbol{b}_1=
	\boldsymbol{e}\times\boldsymbol{\nu}_1$ and $\boldsymbol{b}_2=
	\boldsymbol{e}\times\boldsymbol{\nu}_2$.
	Then by the Frenet-Serre type formula, $(\kappa_{11},\kappa_{12},\kappa_{13})$ and $(\kappa_{21},\kappa_{22},\kappa_{23})$ are
	$$
	\kappa_{11}=-1,\
	\kappa_{12}=0,\
	\kappa_{13}=0;
	$$
	$$
	\kappa_{21}=-\frac{1}{\sqrt{2}},\ 
	\kappa_{22}=\frac{1}{\sqrt{2}},\ 
	\kappa_{23}=0.
	$$
	Set 
	$$
	S_{\boldsymbol{\nu}_1}(u,v)=
	\hat{\gamma}+v\Big(\dfrac{\kappa_{13}\boldsymbol{e}+\kappa_{11}\boldsymbol{b}_1}{\sqrt{\kappa_{13}^2+\kappa_{11}^2}}\Big),\
	S_{\boldsymbol{\nu}_2}(u,v)=
	\hat{\gamma}+v\Big(\dfrac{\kappa_{23}\boldsymbol{e}+\kappa_{21}\boldsymbol{b}_2}{\sqrt{\kappa_{23}^2+\kappa_{21}^2}}\Big).
	$$
	The images of $S_{\boldsymbol{\nu}_1}$ and $S_{\boldsymbol{\nu}_2}$ are the same as those of $f_1$ and $f_2$ respectively, and their unit normal vectors are also the same.
	Moreover, they are glued along $\hat{\gamma}$ as shown in the Figure \ref{fig:ex5}.
	For $S_{\boldsymbol{\nu}_1}$, we can calculate that
	$$
	\beta_{\boldsymbol{\nu}_1}=0.
	$$
	For $S_{\boldsymbol{\nu}_2}$, we can calculate that
	$$
	\beta_{\boldsymbol{\nu}_2}=\frac{1}{2\sqrt{2}},\
	\rho_{\boldsymbol{\nu}_2}=0.
	$$
	This shows that $(f_1,f_2,\hat{\gamma})$ is not only $S_{\boldsymbol{\nu}_1}$-cylindrical glue
	but also $S_{\boldsymbol{\nu}_2}$-conical glue.
\end{example}
\begin{figure}[htbp]
	\centering
	\includegraphics[width=15cm]{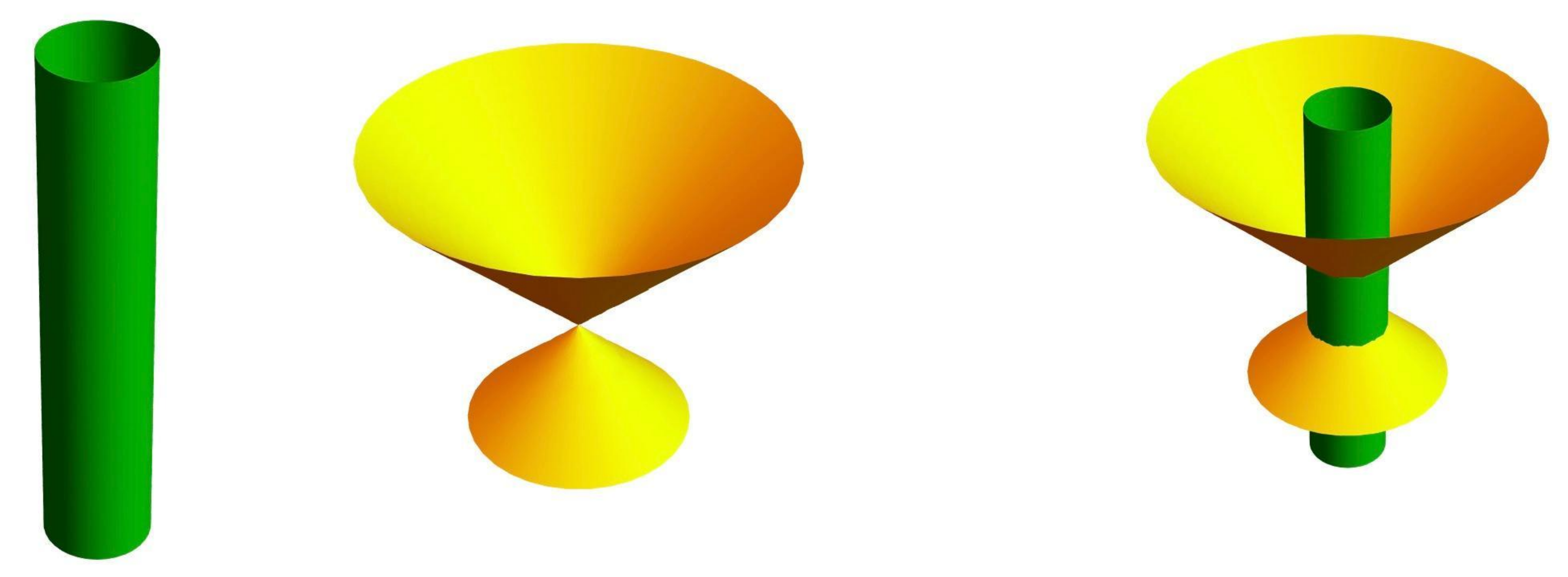}
	\caption{
		From left to right,
		image of $(f_1, S_{\boldsymbol{\nu}_1})$,
		image of $(f_2, S_{\boldsymbol{\nu}_2})$, 
		image of Glued ($S_{\boldsymbol{\nu}_1}, S_{\boldsymbol{\nu}_2}$).}
	\label{fig:ex5}
\end{figure}

\begin{example}
	We give an example of
	$S_{\boldsymbol{\nu}_1}$-cylindrical glue.
	Let us set 
	$\hat{\gamma}(u,0)=(\cos u,\sin u,0)$
	and let us set $f_1$ and $f_2$ by
	$$
	f_1(u,v)=\bigg(\sin(v+\frac{\pi}{2})\cos u,\sin(v+\frac{\pi}{2})\sin u,\cos(v+\frac{\pi}{2})\bigg),
	$$
	$$
	f_2(u,v)=\big(\cos u,v+\sin u,0\big)
	$$
	where the gluing locus is $\hat{\gamma}$. 
	They are shown in the Figure \ref{fig:ex6}.
	Let us set
	$$
	\boldsymbol{e}=\big(-\sin u,\cos u,0\big),
	$$
	$$
	\boldsymbol{\nu}_1=\big(-\cos u,-\sin u,0\big),
	$$
	$$
	\boldsymbol{\nu}_2=\big(0,0,-1\big).
	$$
	Then $\boldsymbol{e}=\hat{\gamma}'$,
	where the length function of $\hat{\gamma}$ is $l(u)=1$.
	The vectors $\boldsymbol{\nu}_1$, $\boldsymbol{\nu}_2$ are the unit normal vectors of $f_1$ and $f_2$ on the gluing locus $\hat{\gamma}$ respectively.
	Moreover, the angle between $\boldsymbol{\nu}_1$
	and $\boldsymbol{\nu}_2$ is $\frac{\pi}{2}$.
	We set $\boldsymbol{b}_1=
	\boldsymbol{e}\times\boldsymbol{\nu}_1$ and $\boldsymbol{b}_2=
	\boldsymbol{e}\times\boldsymbol{\nu}_2$.
	Then by the Frenet-Serre type formula, $(\kappa_{11},\kappa_{12},\kappa_{13})$ and $(\kappa_{21},\kappa_{22},\kappa_{23})$ are
	$$
	\kappa_{11}=1,\
	\kappa_{12}=0,\
	\kappa_{13}=0;
	$$
	$$
	\kappa_{21}=0,\ 
	\kappa_{22}=1,\ 
	\kappa_{23}=0.
	$$
	It can be seen that $(\kappa_{11},\kappa_{12},\kappa_{13})$ and $(\kappa_{21},\kappa_{22},\kappa_{23})$ satisfy \eqref{eq:k123theta} in Lemma \ref{lm:k2-k1}.
	And we set 
	$$
	S_{\boldsymbol{\nu}_1}(u,v)=
	\hat{\gamma}+v\Big(\dfrac{\kappa_{13}\boldsymbol{e}+\kappa_{11}\boldsymbol{b}_1}{\sqrt{\kappa_{13}^2+\kappa_{11}^2}}\Big)
	=\big(\cos u,\sin u, -v\big).
	$$
	Then for $S_{\boldsymbol{\nu}_1}$, we can calculate that
	$$
	\beta_{\boldsymbol{\nu}_1}=0.
	$$
	According to Theorem \ref{thm:4.5}, $S_{\boldsymbol{\nu}_1}$ is a cylinder, 
	as shown in Figure \ref{fig:ex6}.
	This shows that $(f_1,f_2,\hat{\gamma})$ is $S_{\boldsymbol{\nu}_1}$-cylindrical glue.
\end{example}
\begin{figure}[htbp]
	\centering
	\includegraphics[width=15cm]{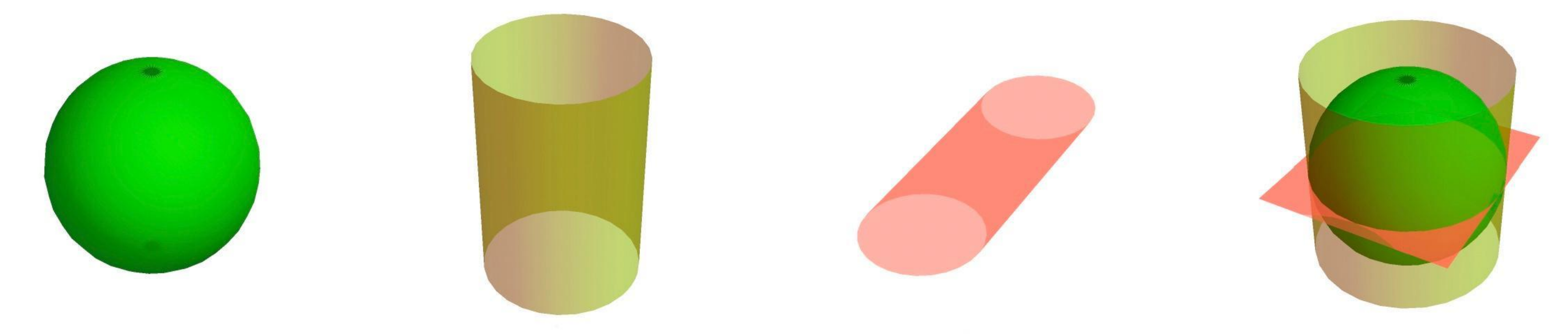}
	\caption{
		From left to right,
		image of $f_1$,
		image of $S_{\boldsymbol{\nu}_1}$,
		image of $f_2$,
		image of Glued ($f_1, S_{\boldsymbol{\nu}_1}, f_2$) along $\hat{\gamma}$.}
	\label{fig:ex6}
\end{figure}

\begin{example}
	We give an example of
	$S_{\boldsymbol{\nu}_2}$-conical glue.
	Let us set
	$\hat{\gamma}(u,0)=\frac{\sqrt{2}}{2}(\cos u,\sin u,1)$
	and let us set $f_1$ and $f_2$ by
	$$
	f_1(u,v)=\bigg(\frac{\sqrt{2}}{2}\cos u,a+\frac{\sqrt{2}}{2}\sin u,\frac{\sqrt{2}}{2}\bigg),
	$$
	$$
	f_2(u,v)=\bigg(\sin(v+\frac{\pi}{4})\cos u,\sin(v+\frac{\pi}{4})\sin u,-\cos(v+\frac{\pi}{4})\bigg)
	$$
	where the gluing locus is $\hat{\gamma}$.
	They are shown in Figure \ref{fig:ex7}.
	Let us set
	$$
	\boldsymbol{e}=\big(-\sin u,\cos u,0\big),
	$$
	$$
	\boldsymbol{\nu}_1=\big(0,0,1\big),
	$$
	$$
	\boldsymbol{\nu}_2=\frac{1}{\sqrt{2}}\big(\cos u,\sin u,-1\big).
	$$
	Then $\hat{\gamma}'=\frac{\sqrt{2}}{2}\boldsymbol{e}$,
	where the length function of $\hat{\gamma}$ is $l(u)=\sqrt{2}/2$.
	The vectors $\boldsymbol{\nu}_1$, $\boldsymbol{\nu}_2$ are the unit normal vectors of $f_1$ and $f_2$ on the gluing locus $\hat{\gamma}$ respectively.
	Moreover, the angle between $\boldsymbol{\nu}_1$
	and $\boldsymbol{\nu}_2$ is $3\pi/2$.
	We set $\boldsymbol{b}_1=
	\boldsymbol{e}\times\boldsymbol{\nu}_1$ and $\boldsymbol{b}_2=
	\boldsymbol{e}\times\boldsymbol{\nu}_2$.
	Then by the Frenet-Serre type formula, $(\kappa_{11},\kappa_{12},\kappa_{13})$ and $(\kappa_{21},\kappa_{22},\kappa_{23})$ are
	$$
	\kappa_{11}=0,\
	\kappa_{12}=-1,\
	\kappa_{13}=0;
	$$
	$$
	\kappa_{21}=-\frac{1}{\sqrt{2}},\ 
	\kappa_{22}=\frac{1}{\sqrt{2}},\ 
	\kappa_{23}=0.
	$$
	It can be seen that $(\kappa_{11},\kappa_{12},\kappa_{13})$ and $(\kappa_{21},\kappa_{22},\kappa_{23})$ satisfy \eqref{eq:k123theta} in Lemma \ref{lm:k2-k1}.
	And we set 
	$$
	S_{\boldsymbol{\nu}_2}(u,v)=
	\hat{\gamma}+v\Big(\dfrac{\kappa_{23}\boldsymbol{e}+\kappa_{21}\boldsymbol{b}_2}{\sqrt{\kappa_{23}^2+\kappa_{21}^2}}\Big)
	=\Big((v+\frac{\sqrt{2}}{2})\cos u,(v+\frac{\sqrt{2}}{2})\sin u,v+\frac{\sqrt{2}}{2}\Big).
	$$
	Then for $S_{\boldsymbol{\nu}_2}$, we can calculate that
	$$
	\beta_{\boldsymbol{\nu}_2}=\frac{1}{2\sqrt{2}},\
	\rho_{\boldsymbol{\nu}_2}=0.
	$$
	According to Theorem \ref{thm:4.5}, $S_{\boldsymbol{\nu}_2}$ is a cone, 
	as shown in Figure \ref{fig:ex7}.
	This shows that $(f_1,f_2,\hat{\gamma})$ is $S_{\boldsymbol{\nu}_2}$-conical glue.
\end{example}
\begin{figure}[htbp]
	\centering
	\includegraphics[width=15cm]{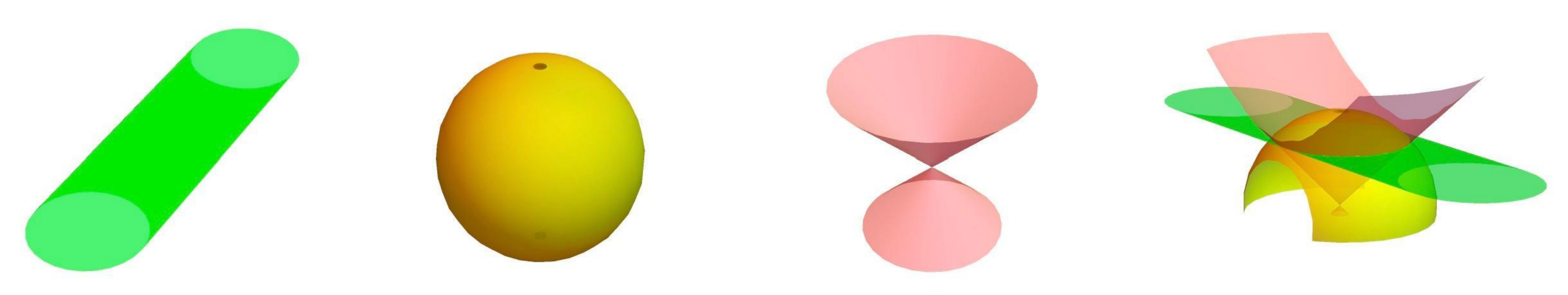}
	\caption{
		From left to right,
		image of $f_1$,
		image of $S_{\boldsymbol{\nu}_1}$,
		image of $f_2$,
		image of Glued ($f_1, S_{\boldsymbol{\nu}_1}, f_2$) along $\hat{\gamma}$.}
	\label{fig:ex7}
\end{figure}

\begin{example}
	We give an example of
	$S_{\boldsymbol{\nu}_2}$-cuspidal edgy glue 
	and it is obtained by rotating the unit normal vector of the plane. 
	Let us set
	$\hat{\gamma}(u,0)=(u^2,u^3,0)$
	and let us set $f_1$ and $f_2$ by
	$$
	f_1(u,v)=\big(u^2,u^3+v,0\big),\
	f_2(u,v)=\big(u^2+v,u^3+\frac{3}{2}uv+v,v\big)
	$$
	where the gluing locus is $\hat{\gamma}$.
	They are show in Figure \ref{fig:ex3}.
	Let us set
	$$
	\boldsymbol{e}=\frac{1}{\sqrt{4+9u^2}}\big(2,3u,0\big),
	$$
	$$
	\boldsymbol{\nu}_1=\big(0,0,1\big),
	$$
	$$
	\boldsymbol{\nu}_2=\frac{1}{\sqrt{8+9u^2}}\big(3u,-2,2\big).
	$$
	Then $\hat{\gamma}'=l(u)\boldsymbol{e}$,
	where the length function of $\hat{\gamma}$ is $l(u)=u\sqrt{4+9u^2}$.
	The vectors $\boldsymbol{\nu}_1$, $\boldsymbol{\nu}_2$ are the unit normal vectors of $f_1$ and $f_2$ on the gluing locus $\hat{\gamma}$ respectively.
	We set $\boldsymbol{b}_1=
	\boldsymbol{e}\times\boldsymbol{\nu}_1$ and $\boldsymbol{b}_2=
	\boldsymbol{e}\times\boldsymbol{\nu}_2$.
	Then by the Frenet-Serre type formula, $(\kappa_{11},\kappa_{12},\kappa_{13})$ and $(\kappa_{21},\kappa_{22},\kappa_{23})$ are
	$$
	\kappa_{11}=0,\
	\kappa_{12}=-\frac{6}{4+9u^2},\
	\kappa_{13}=0;
	$$
	$$
	\kappa_{21}=-\frac{6}{\sqrt{(4+9u^2)(8+9u^2)}},\ 
	\kappa_{22}=-\frac{12}{(4+9u^2)\sqrt{8+9u^2}},\ 
	\kappa_{23}=\frac{18u}{\sqrt{4+9u^2}(8+9u^2)}.
	$$
	Let the function $\theta(u)$ be the angle between $\boldsymbol{\nu}_1$
	and $\boldsymbol{\nu}_2$. 
	We set
	$$
	\sin\theta
	=\sqrt{\frac{4+9u^2}{8+9u^2}},\
	\cos\theta
	=\frac{2}{\sqrt{8+9u^2}},\
	\theta'=\frac{18u}{\sqrt{4+9u^2}(8+9u^2)}
	$$
	It can be seen that $(\kappa_{11},\kappa_{12},\kappa_{13})$ and $(\kappa_{21},\kappa_{22},\kappa_{23})$ satisfy \eqref{eq:k123theta} in Lemma \ref{lm:k2-k1}.
	We set
	$$
	S_{\boldsymbol{\nu}_2}(u,v)=
	\hat{\gamma}+v\Big(\dfrac{\kappa_{23}\boldsymbol{e}+\kappa_{21}\boldsymbol{b}_2}{\sqrt{\kappa_{23}^2+\kappa_{21}^2}}\Big)
	=\big(u^2,u^3+v,v\big).
	$$
	Then $S_{\boldsymbol{\nu}_2}$ and $f_2$ have the same image
	and unit normal vector.
	Moreover, they are glued along $\hat{\gamma}$ and $f_1$.
	It is shown in Figure \ref{fig:ex3}.
	For the tangent developable surface $S_{\boldsymbol{\nu}_2}$,
	when $u=0$, $l(u)=0$,
	then there is a singularity at $(0,0)$.
	And we can calculate
	$$
	l'(u)=\frac{4+18u^2}{\sqrt{4+9u^2}}.
	$$
	Then we obtain 
	$(\kappa_{11}\cos\theta+\kappa_{12}\sin\theta)l'|_{u=0}\neq0$.
	By Corollary \ref{coro:v2}, this show that
	$S_{\boldsymbol{\nu}_2}$ is cuspidal edge at $(0,0)$.
	In conclusion, $(f_1,f_2,\hat{\gamma})$ is $S_{\boldsymbol{\nu}_2}$-cuspidal edgy glue.
\end{example}
\begin{figure}[htbp]
	\centering
	\includegraphics[width=15cm]{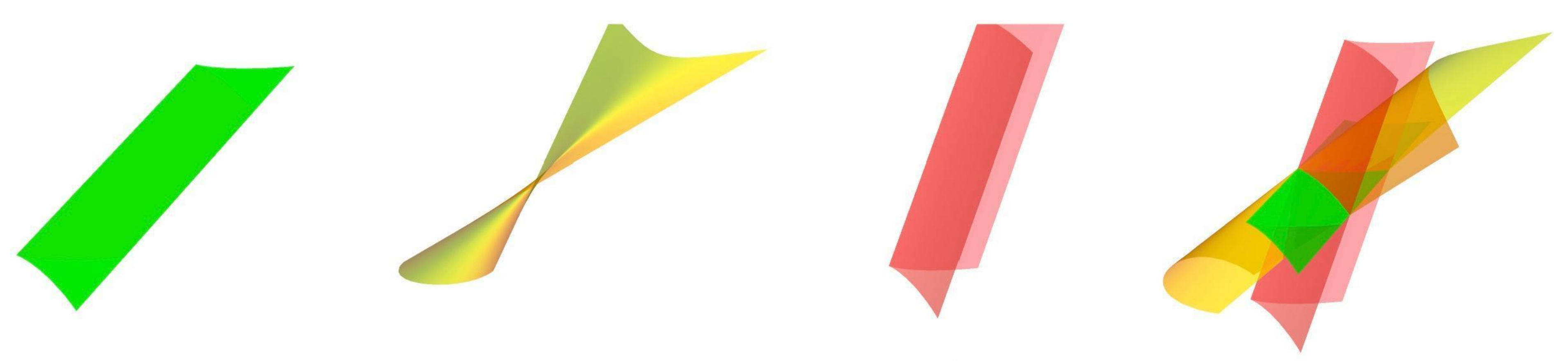}
	\caption{
		From left to right,
		image of $f_1$,
		image of $S_{\boldsymbol{\nu}_1}$,
		image of $f_2$,
		image of Glued ($f_1, S_{\boldsymbol{\nu}_1}, f_2$) along $\hat{\gamma}$.}
	\label{fig:ex3}
\end{figure}

\begin{example}
	We give an example of
	$S_{\boldsymbol{\nu}_2}$-swallowtailed glue 
	and it is obtained by rotating the unit normal vector of the plane. 
	Let us set
	$\hat{\gamma}(u,0)=(0, 4u^3, 3u^4)$
	and let us set $f_1$ and $f_2$ by
	$$
	f_1(u,v)=\big(0, 4u^3+v, 3u^4\big),\
	f_2(u,v)=\big(v, 4u^3+2uv+v^2, 3u^4+u^2v-v^2\big).
	$$
	where the gluing locus is $\hat{\gamma}$.
	They are show in Figure \ref{fig:ex4}.
	Let us set
	$$
	\boldsymbol{e}=\frac{1}{\sqrt{1+u^2}}\big(0,1,u\big),
	$$
	$$
	\boldsymbol{\nu}_1=\big(0,0,1\big),
	$$
	$$
	\boldsymbol{\nu}_2=\frac{1}{\sqrt{1+u^2+u^4}}\big(-u^2,u,-1\big).
	$$
	Then $\hat{\gamma}'=l(u)\boldsymbol{e}$,
	where the length function of $\hat{\gamma}$ is $l(u)=12u^2\sqrt{1+u^2}$.
	The vectors $\boldsymbol{\nu}_1$, $\boldsymbol{\nu}_2$ are the unit normal vectors of $f_1$ and $f_2$ on the gluing locus $\hat{\gamma}$ respectively.
	We set $\boldsymbol{b}_1=
	\boldsymbol{e}\times\boldsymbol{\nu}_1$ and $\boldsymbol{b}_2=
	\boldsymbol{e}\times\boldsymbol{\nu}_2$.
	Then by the Frenet-Serre type formula, $(\kappa_{11},\kappa_{12},\kappa_{13})$ and $(\kappa_{21},\kappa_{22},\kappa_{23})$ are
	$$
	\kappa_{11}=\frac{1}{(1+u^2)^\frac{3}{2}},\
	\kappa_{12}=0,\
	\kappa_{13}=0;
	$$
	$$
	\kappa_{21}=-\frac{1}{\sqrt{(1+u^2)(1+u^2+u^4)}},\ 
	\kappa_{22}=\frac{u^2}{(1+u^2)\sqrt{1+u^2+u^4}},\ 
	\kappa_{23}=\frac{u(2+u^2)}{\sqrt{1+u^2}(1+u^2+u^4)}.
	$$
	Let the function $\theta(u)$ be the angle between $\boldsymbol{\nu}_1$
	and $\boldsymbol{\nu}_2$.
	We set
	$$
	\sin\theta
	=\sqrt{\frac{1+u^2}{1+u^2+u^4}},\
	\cos\theta
	=\frac{u^2}{\sqrt{1+u^2+u^4}},\
	\theta'=\frac{u(2+u^2)}{\sqrt{1+u^2}(1+u^2+u^4)}.
	$$
	It can be seen that $(\kappa_{11},\kappa_{12},\kappa_{13})$ and $(\kappa_{21},\kappa_{22},\kappa_{23})$ satisfy \eqref{eq:k123theta} in Lemma \ref{lm:k2-k1}.
	We set
	$$
	S_{\boldsymbol{\nu}_2}(u,v)=
	\hat{\gamma}+v\Big(\dfrac{\kappa_{23}\boldsymbol{e}+\kappa_{21}\boldsymbol{b}_2}{\sqrt{\kappa_{23}^2+\kappa_{21}^2}}\Big)
	=\big(v, 4u^3+2uv, 3u^4+u^2v\big).
	$$
	Then $S_{\boldsymbol{\nu}_2}$ and $f_2$ have the same image
	and unit normal vector.
	Moreover, they are glued along $\hat{\gamma}$ and $f_1$.
	It is shown in Figure \ref{fig:ex4}.
	For the tangent developable surface $S_{\boldsymbol{\nu}_2}$,
	when $u=0$, $l(u)=0$,
	then there is a singularity at $(0,0)$.
	And we can calculate
	$$
	l'(u)=\frac{12u(2+3u^2)}{\sqrt{1+u^2}},\
	l''(u)=\frac{12(2+9u^2+6u^4)}{(1+u^2)^\frac{3}{2}}.
	$$
	Then we obtain $l'|_{u=0}=0$ and  $(\kappa_{11}\cos\theta+\kappa_{12}\sin\theta)l''|_{u=0}=-24$.
	By Corollary \ref{coro:v2}, this show that
	$S_{\boldsymbol{\nu}_2}$ is swallowtailed at $(0,0)$.
	In conclusion, $(f_1,f_2,\hat{\gamma})$ is $S_{\boldsymbol{\nu}_2}$-swallowtailed glue.
\end{example}
\begin{figure}[htbp]
	\centering
	\includegraphics[width=15cm]{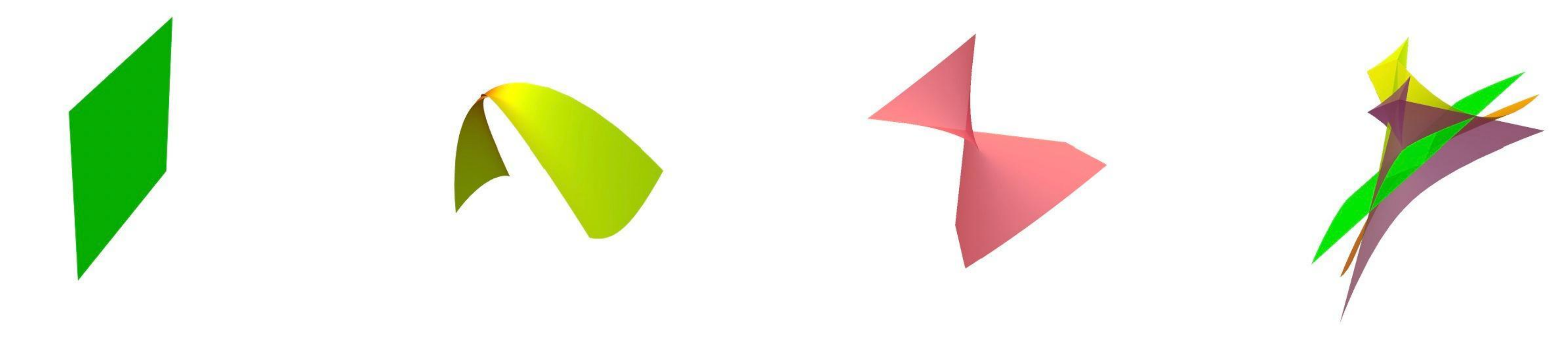}
	\caption{
		From left to right,
		image of $f_1$,
		image of $S_{\boldsymbol{\nu}_1}$,
		image of $f_2$,
		image of Glued ($f_1, S_{\boldsymbol{\nu}_1}, f_2$) along $\hat{\gamma}$.}
	\label{fig:ex4}
\end{figure}

\end{document}